\documentclass[11pt, reqno]{amsart}

\usepackage{amsmath}
\usepackage{fullpage}
\usepackage{amssymb}
\usepackage{amsthm}
\usepackage{bbm}
\usepackage{subfigure}
\usepackage{float}
\usepackage{graphicx}
\usepackage{enumitem}
\usepackage{epstopdf}
\usepackage{xcolor}
\usepackage[hidelinks]{hyperref}

\usepackage{titlesec} 
\titleformat{\section}{\vskip10pt\large\bfseries}{\thesection.}{0.5em}{\centering\vspace{5pt}}
\titleformat{\subsection}{\vskip10pt\normalsize\bfseries}{\thesubsection.}{0.5em}{}
\titleformat{\subsubsection}{\vskip10pt\normalsize\bfseries}{\thesubsection.}{0.5em}{}
\usepackage{lipsum,titletoc}
\titlecontents{section}[2.5em]{\rmfamily}
{\contentslabel{2.3em}} 
{\hspace*{-2.3em}} {\titlerule*[1pc]{.}\contentspage}
\titlecontents{subsection} [4.8em]{\rmfamily}             
{\contentslabel{2.3em}} 
{\hspace*{-2.3em}} {\titlerule*[1pc]{.}\contentspage}

\allowdisplaybreaks

\newtheorem{lemma}{Lemma}[section]
\newtheorem{theorem}[lemma]{Theorem}
\newtheorem{proposition}[lemma]{Proposition}
\newtheorem{remark}[lemma]{Remark}

\newtheorem{assumption}[lemma]{Assumption}
\newcommand{\fe}{{\rm e}}
\newcommand{\im}{{\rm i}}
\numberwithin{equation}{section}
\DeclareMathOperator{\sgn}{sgn}
\DeclareMathOperator{\tv}{TV}
\DeclareMathOperator{\bv}{BV}
\graphicspath{{./fig/}}

\title{A vanishing-viscosity minimizing movement\\ Fourier spectral method for scalar conservation laws}

\author{Lun Ji}
\author{\,\,Buyang Li}
\author{\,Fangyan Yao}
\address{Department of Applied Mathematics, The Hong Kong Polytechnic University, Hung Hom, Hong Kong}
\email{lun-422.ji@polyu.edu.hk, buyang.li@polyu.edu.hk, fangyan.yao@polyu.edu.hk}
\thanks{This work was supported in part by the National Natural Science Foundation of China (grant No.~12525111).}

\begin{document}


\begin{abstract}
We propose a vanishing-viscosity minimizing movement (VVMM) Fourier spectral method for scalar hyperbolic conservation laws under periodic boundary conditions. The method is a fully discrete space--time residual-minimizing scheme, Fourier spectral in space and polynomial in time, stabilized by vanishing viscosity and heat-kernel postprocessing. A basic analytical difficulty in Fourier discretizations of nonlinear conservation laws is the interaction between Fourier projection and the nonlinear flux; the resulting projection residual contains unresolved high-frequency components. VVMM addresses this issue by selecting, on each time slab, the numerical solution that minimizes {a penalized cut-off viscous conservation laws residual} and then applies heat-kernel postprocessing to damp unresolved modes. We prove a fully discrete $L^1$ error estimate, in any fixed space dimension, that separates the vanishing-viscosity error, the temporal interpolation error, and the spatial Fourier truncation error. {Under exact minimization of the penalized cut-off residual and suitable parameter coupling,} the estimate yields, for each fixed $\gamma>0$, an $L^1$ convergence rate $N^{-1/2+\gamma}$, with a constant depending on $\gamma$. Numerical experiments in one and two dimensions illustrate stable shock capturing without visible Gibbs oscillations and convergence behavior consistent with the theory.
\end{abstract}

\maketitle
\vspace{-2\baselineskip}

\section{Introduction}

Scalar hyperbolic conservation laws provide a standard mathematical setting for studying nonlinear transport, finite-speed propagation, and shock formation in continuum models such as fluid dynamics, traffic flow, gas dynamics, and sedimentation; see, for example, \cite{Dafermos2016,LeVeque2002}. This paper develops a Fourier spectral approximation framework for the periodic scalar problem
\begin{equation}\label{periodic}
\begin{aligned}
&\partial_t u + \nabla \cdot f(u)=0, \qquad (t,x)\in[0,T]\times\mathbb{T}^d,\\
&u(0,x)=u_0(x)\in L^\infty\cap\bv(\mathbb{T}^d).
\end{aligned}
\end{equation}
Here \(u=u(t,x)$ is the unknown conserved quantity and $f:\mathbb{R}\to\mathbb{R}^d$ is a nonlinear flux. The periodic formulation is the setting of the analysis and convergence theorem below. It is also compatible with localized Cauchy problems on $\mathbb R^d$ over finite time intervals: by finite speed of propagation, one may choose a sufficiently large periodic box so that periodic copies do not interact with the region of interest up to time $T$, and then rescale the box to $\mathbb T^d$.

Although \eqref{periodic} is simple in form, smooth initial data may generate shocks in finite time. After shock formation, classical solutions generally cease to exist and weak solutions are not unique. The relevant solution concept is therefore the entropy solution, selected by admissibility criteria such as entropy inequalities or the vanishing-viscosity principle \cite{Dafermos2016,Kruzkov1970}. We use this principle as the continuous reference mechanism. For $\varepsilon>0$, the corresponding viscous problem is
\begin{equation}\label{vis}
\begin{aligned}
&\partial_t u^\varepsilon+\nabla\cdot f(u^\varepsilon)
=\varepsilon\Delta u^\varepsilon,
\qquad (t,x)\in[0,T]\times\mathbb T^d,\\
&u^\varepsilon(0,x)=u_0(x)\in L^\infty\cap\bv(\mathbb{T}^d) .
\end{aligned}
\end{equation}
It is standard that $u^\varepsilon$ converges to the entropy solution as $\varepsilon\to0$. In particular, for $L^\infty\cap\bv$ data one has the Kuznetsov-type estimate
\begin{equation}\label{visconv}
\|u(t)-u^\varepsilon(t)\|_{L^1}\le C\varepsilon^{1/2},
\qquad 0\le t\le T,
\end{equation}
where $C$ depends on $T$, $f$, and $u_0$, but not on $t$ or $\varepsilon$.

From the numerical perspective, shock formation introduces another difficulty beyond the selection of the entropy solution. Specifically, numerical approximations may generate nonphysical oscillations near the shock. This is particularly transparent for Fourier approximations, where it appears as the classical Gibbs effect. However, the underlying difficulty is more general: an appropriate stabilization mechanism is required to control nonphysical oscillations around  discontinuities, especially for high-order approximations.

A large body of numerical methods has therefore been developed to address these difficulties. Monotone finite-volume schemes, Godunov-type methods, ENO/WENO reconstructions, discontinuous Galerkin methods with suitable limiters, and related shock-capturing algorithms are robust and efficient in practice \cite{CockburnShu1998,Godunov1959,HartenLaxVanLeer1983,JiangShu1996,LeVeque2002,ShuOsher1988}. {Their stability mechanisms are typically local: numerical fluxes, nonlinear reconstructions, entropy fixes, and limiters enforce monotonicity, entropy admissibility, or nonlinear bounds, often through numerical dissipation that suppresses nonphysical oscillations near shocks.} For monotone schemes, Kuznetsov-type arguments give classical $L^1$ error estimates; in particular, the one-half order rate is obtained in one space dimension \cite{Kuznetsov1976}, and related multidimensional estimates are available for flux-splitting monotone schemes on Cartesian-product grids \cite{CockburnGremaud1997}. These methods are indispensable computational tools. Their stability theory, however, is built on local structures such as monotonicity,  conservation across cell interfaces, discrete entropy inequalities, and, in some settings, bounded variation estimates. These structures are not directly available for global spectral residual minimization. A large high-order literature further develops maximum principle, positivity, invariant-region, and bound-preserving mechanisms for finite-volume, finite-difference WENO, and discontinuous Galerkin methods through local limiting or flux correction; representative works include \cite{VanderVegtXiaXu2019,WuShu2023,XiongQiuXu2013,QZhangShu2004,QZhangShu2010,ZhangShu2010,ZhangShu2010Euler,ZhangShu2011}. These methods preserve key nonlinear bounds through local constraints on cell averages, reconstructed polynomials, or numerical fluxes. The present work pursues a different mechanism: the approximation is selected globally on each time slab by minimizing a viscous residual over a Fourier trial space, and stability is measured against a heat-postprocessed vanishing-viscosity reference flow rather than enforced by a local limiter.

Fourier spectral methods offer a different type of approximation. They use global basis functions and can be highly accurate for smooth solutions. For conservation laws, spectral viscosity methods provide a classical way to combine Fourier accuracy with nonlinear stability mechanisms \cite{ChenDuTadmor1993,MadayTadmor1989,Schochet1990,Tadmor1989,Tadmor1993}. In particular, Chen, Du, and Tadmor proved convergence and $L^1$ one-half order error estimates for multidimensional scalar periodic conservation laws within a semi-discrete Fourier spectral-viscosity framework \cite{ChenDuTadmor1993}. In that approach, a prescribed spectral-viscosity operator supplies the entropy dissipation missing from the standard Fourier method while remaining spectrally small on the resolved scales. Vanishing-viscosity approximations have also been used to construct and analyze semi-implicit Fourier spectral schemes for the incompressible Euler equations \cite{ChengLuoWang2024}.

The main novelty of this paper is a fully discrete vanishing-viscosity minimizing movement (VVMM) Fourier spectral method for the heat-postprocessed viscous reference flow. On each time slab, the approximation is defined as an exact minimizer of a viscous residual over a finite-dimensional space--time trial space, Fourier spectral in space and polynomial in time. Between successive slabs, a heat-kernel postprocessing operator is applied. Thus the method combines three ingredients: the vanishing-viscosity selection principle, global space--time residual minimization, and Fourier-diagonal heat smoothing.

For the unconstrained minimization used in the theoretical formulation, a smooth cut-off flux and a penalty measuring violations of the initial-data range are included. These devices make the finite-dimensional optimization problem well-posed without changing the target entropy solution or the final convergence rate.

This formulation leads to a concrete residual mechanism for treating the interaction between Fourier projection and the nonlinear flux. In general,
$$
P_N\nabla\cdot f(u)\neq \nabla\cdot f(P_Nu).
$$
The unresolved part of this residual is produced by the spatial Fourier truncation, while the time-slab construction may propagate high-frequency components from one slab endpoint to the next. The usual $L^1$-contraction or maximum principle arguments for the continuous equation do not apply directly after spectral truncation. The role of the vanishing-viscosity term and heat-kernel postprocessing is to restore, between successive slabs, a uniform smoothing scale on which the residual estimates can be closed.

This combination of vanishing viscosity, residual minimization, and heat-semigroup postprocessing gives a direct way to connect entropy-solution stability with Fourier spectral approximation. Writing $\ell_k:=1+\log k$, it leads to a fully discrete error estimate in $L^1$ of the form
$$
\|u_N^\varepsilon-u\|_{L^1}\le C\varepsilon^{1/2}+C\varepsilon^{-m-3\delta}\tau^m \ell_kk^{-m}
+C\varepsilon^{-m-3\delta}N^{-m},
$$
up to the precise norms and parameter conditions stated later. The three terms represent, respectively, the vanishing-viscosity, temporal interpolation, and Fourier truncation errors. The factor $\ell_k$ accounts for the logarithmic growth of the Chebyshev interpolation constants. Under suitable choices of $\varepsilon$, $\tau$, $N$, $k$, and $m$, this yields an explicit algebraic convergence rate in $L^1$. More explicitly, we take
$$
\varepsilon=N^{-\frac{1}{1+\eta}}\qquad\text{equivalently}\qquad N=\varepsilon^{-1-\eta},
$$
with an arbitrarily small fixed $\eta>0$, and choose $\tau\sim\varepsilon$. The temporal degree is chosen so that $\ell_k k^{-m}\le N^{-d-\vartheta}$ for a fixed $\vartheta>0$. The first term is then $O(N^{-1/(2(1+\eta))})$. By taking $\eta$ small relative to a prescribed $\gamma>0$, and then choosing $m$, $k$, and the auxiliary small exponents so that the temporal and spatial truncation terms are of the same or smaller order, the estimate yields
$$
\|u_N^\varepsilon-u\|_{L^1}\le C_\gamma N^{-1/2+\gamma}, \qquad \gamma>0.
$$

The main contributions of this paper are summarized as follows:

\begin{itemize}
\item \textbf{A residual-minimizing spectral framework.}
We construct a fully discrete VVMM Fourier spectral method by minimizing a viscous space--time residual on each time slab and applying heat-kernel postprocessing between successive slabs. In the theoretical formulation, a smooth cut-off flux and a penalty measuring violations of the initial-data range are used to make the unconstrained finite-dimensional minimization problem well-posed.

\item \textbf{Control of nonlinear projection residuals.} The analysis isolates the noncommutativity between Fourier projection and the nonlinear flux and shows how the minimizing movement step, together with heat smoothing, yields residual bounds that remain uniform across time slabs.

\item \textbf{A multidimensional $L^1$ error estimate.} We prove a fully discrete convergence estimate that holds in any fixed spatial dimension. The proof uses the $L^1$-stability of viscous conservation laws and does not require one-dimensional monotonicity or total-variation arguments beyond the basic entropy-solution framework.

\item \textbf{Numerical validation.} Numerical experiments in one and two dimensions illustrate stable shock capturing with no pronounced Gibbs-type oscillations in the reported tests and support the convergence behavior predicted by the theory.
\end{itemize}

The rest of the paper is organized as follows. Section~\ref{sectionmethod} defines the VVMM Fourier spectral method and states the main convergence theorem. {Section~\ref{sectionprop} collects analytic estimates for the viscous equation and proves the $O(\varepsilon^{1/2})$ error estimate for the heat-postprocessed viscous reference flow.} Section~\ref{sectionerroranal} proves the space--time residual estimate and the main theorem. Section~\ref{sectionnumexp} presents numerical experiments in one and two dimensions.

\section{The numerical method}\label{sectionmethod}

Fix a final time $T>0$ and a time-step length $\tau$ such that $M:=T/\tau\in\mathbb N$. We introduce the uniform partition
$$
0 = T_0 < T_1 < \cdots < T_M=T,\qquad T_i:= i\tau,
$$
and define the half-open subintervals
$$
I_i := (T_i, T_{i+1}],\qquad i=0,\ldots,M-1.
$$

On each time slab $I_i\times\mathbb{T}^d$, the approximation is chosen from a
finite-dimensional space--time trial space. For given integers $k,N\ge1$, set
\begin{equation}\label{eq:trialspace}
S_{k,N}=\operatorname{span}\left\{
\sum_{j=0}^k t^j v_j \; : \; v_j \in X
\right\},\qquad
X = \operatorname{span}\left\{
\sum_{|m|_\infty\le N} c_m \fe^{im\cdot x}
\right\},
\end{equation}
where $ |m|_\infty=\max_{1\le j\le d}|m_j|$. Thus $N$ is the
Fourier frequency cutoff, not the number of physical grid points. We denote by
$P_N$ the Fourier projection onto the spatial modes $|m|_\infty\le N$.
For each $i=0,\ldots,M-1$, define the left endpoint data
$$
w_0=P_N\fe^{\varepsilon^2\Delta}u_0,\qquad
w_i=\fe^{\varepsilon^2\Delta}u_{N,i-1}^\varepsilon(T_i),\quad i\ge1,
$$
and the endpoint trial set
$$
\widehat{\mathcal A}_{i,k,N}:=
\left\{
V\in S_{k,N}:\ V(T_i)=w_i
\right\}.
$$
 Moreover, we denote
\begin{equation}\label{u-u+}
u_-:=\operatorname*{ess\,inf}_{x\in\mathbb T^d}u_0(x),\qquad
u_+:=\operatorname*{ess\,sup}_{x\in\mathbb T^d}u_0(x).
\end{equation}
By the maximum principle, both the entropy solution to \eqref{periodic} and the viscous solution to \eqref{vis} remain in $[u_-,u_+]$ (see also Lemma~\ref{lemmaxpr} and Remark~\ref{rem2}). 
 
\begin{assumption}\label{localsmooth}
There exists a fixed constant $\kappa>0$, such that the nonlinear flux $f$ is smooth on the bounded interval $[u_- -\kappa,u_+ +\kappa]$.
\end{assumption}
 In the finite-dimensional minimization problem, this local assumption is implemented by using  a smooth cut-off flux $\widetilde f$ such that $\widetilde f=f$ on $[u_- -\kappa/2,u_+ +\kappa/2]$, with the cut-off supported inside $[u_- -\kappa,u_+ +\kappa]$,  and a penalty measuring violations of the initial-data range. The cutoff is therefore only a technical device for the unconstrained finite-dimensional minimization and does not modify the conservation law on the range attained by the exact solution.

We then define $u_{N,i}^\varepsilon\in S_{k,N}$ as a minimizer of the penalized cut-off viscous space--time residual over this endpoint trial set:
\begin{equation}\label{eq:min}
u_{N,i}^\varepsilon\in\operatorname*{arg\,min}_{V\in{\widehat{\mathcal A}_{i,k,N}}}\left[\Big\|\partial_t V + \nabla\cdot {\widetilde f(V)} - \varepsilon \Delta V\Big\|_{L^{1}(I_i\times\mathbb{T}^d)}+\Big\|\phi(V)\Big\|_{L^1(I_i\times\mathbb{T}^d)}\right],
\end{equation}
where $\phi(V)=\max\{0,V-u_+,u_--V\}$. The set $\widehat{\mathcal A}_{i,k,N}$ is a nonempty closed affine subspace of the finite-dimensional space $S_{k,N}$. Since $\widetilde f$ is smooth with globally bounded derivatives, the map
$$
V\mapsto \partial_tV+\nabla\cdot \widetilde f(V)-\varepsilon\Delta V
$$
is continuous on $S_{k,N}$. The penalty is also continuous. With $B=\max\{|u_-|,|u_+|\}$, the pointwise bound $\phi(V)\ge |V|-B$ yields
$$
\|\phi(V)\|_{L^1(I_i\times\mathbb T^d)}\ge \|V\|_{L^1(I_i\times\mathbb T^d)}-B\tau|\mathbb T^d|.
$$
By equivalence of norms in the finite-dimensional trial space, the continuous objective is coercive and therefore attains its minimum on the closed affine trial set.

The convergence analysis in this paper is for the ideal discrete object defined by \eqref{eq:min}, namely an exact global minimizer of the penalized cut-off continuous space--time residual over the finite-dimensional trial space.

By construction, the minimizers satisfy the endpoint relation
\begin{equation}\label{eq:continuity}
u_{N,0}^\varepsilon(T_0) = P_N\fe^{\varepsilon^2\Delta} u_0,\qquad
u_{N,i}^\varepsilon(T_i) = \fe^{\varepsilon^2\Delta}u_{N,i-1}^\varepsilon(T_i),\quad i=1,\ldots,M-1,
\end{equation}

We then assemble the piecewise-defined approximation on $[0,T]$ by
\begin{equation}\label{eq:globalUN}
u_N^\varepsilon(t,\cdot) := u_{N,i}^\varepsilon(t,\cdot), \qquad t\in I_i,
\end{equation}
We additionally set $u_N^\varepsilon(0)=w_0$. At each internal grid time $T_i$, the assembled value is the left trace $u_{N,i-1}^\varepsilon(T_i)$; the next slab starts from its heat-smoothed right trace, as prescribed in \eqref{eq:continuity}.

Throughout the paper, $I_k$ denotes degree-$k$ interpolation at the Chebyshev--Lobatto nodes
\begin{equation}\label{cheby}
t_{i,j}=T_i+\frac{\tau}{2}\left(1-\cos\frac{j\pi}{k}\right),\qquad j=0,\ldots,k.
\end{equation}
In particular, $(I_kw)(T_i)=w(T_i)$. We write $\ell_k=1+\log k$ for the logarithmic factor in the interpolation estimates. The operator $P_N$ denotes the Fourier projection in space. The notation $a+$ and $a-$ means $a+\vartheta$ and $a-\vartheta$, respectively, for an arbitrarily small $\vartheta>0$. Constants may depend on $T$, $d$, {$\widetilde f$}, $u_0$, $m$, {$\kappa$}, the fixed exponent $\delta>0$, and such arbitrarily small exponents, but are independent of $N$, $k$, $\tau$, $\varepsilon$, and the time-slab index unless otherwise stated. Moreover, due to the smoothing properties of the viscosity (see Proposition~\ref{proputxn}), in the following, we shall freely switch between $\bv$- and $W^{1,1}$-type estimates since they coincide for smooth (even not uniform with respect to $\varepsilon$) solutions.

\begin{theorem}\label{thm}
Let $u$ be the exact solution to \eqref{periodic}  with $f$ satisfying Assumption~\ref{localsmooth}, and let $u_N^\varepsilon$ be the numerical solution defined above. For every fixed $\gamma>0$, the parameters can be chosen so that, with
$$
N=\varepsilon^{-1-\eta}, \qquad \tau\sim \varepsilon,
\qquad \ell_k k^{-m}\le N^{-d-},
$$
one has
$$
\|u_N^\varepsilon(t)-u(t)\|_{L^1(\mathbb{T}^d)}
\le C_\gamma N^{-\frac12+\gamma}, \qquad 0\le t\le T,
$$
where $C_\gamma$ is independent of $N$, $k$, $\tau$, and $\varepsilon$.
\end{theorem}

\begin{remark}\label{rem:parameter-coupling}
The parameter assumptions in Theorem~\ref{thm} have distinct roles. The condition $N=\varepsilon^{-1-\eta}$ means that the Fourier cutoff is slightly above the viscous scale, so that $N\varepsilon\to\infty$ and the heat-kernel tail beyond the resolved modes is negligible. The relation $\tau\sim\varepsilon$ balances the growth of time derivatives of the viscous solution with the factor $\tau^m$ in the temporal interpolation error. The integer $m$ is then chosen sufficiently large so that the spatial truncation error and the regularity bootstrap can absorb the factors depending on $d$, $\eta$, and $\delta$. Finally, $\ell_kk^{-m}\le N^{-d-}$, i.e. $\ell_kk^{-m}\le N^{-d-\vartheta}$ for an arbitrarily small fixed $\vartheta>0$, is the temporal resolution condition used in the bootstrap after heat postprocessing. For a prescribed $\gamma>0$, the choices are made in the following order: first choose $\eta=\eta(\gamma)>0$ sufficiently small, then choose $m=m(\gamma,d)$ sufficiently large, then choose $\delta=\delta(\gamma)>0$ sufficiently small, and finally choose $k\ge m+1$ so that $\ell_kk^{-m}\le N^{-d-\vartheta}$. The constant $C_\gamma$ may depend on $T$, $d$, $f$, $u_0$, and these auxiliary choices, but not on $N$, $k$, $\tau$, or $\varepsilon$. Here $m$ denotes the derivative order used in the interpolation estimate, while $k$ is the temporal polynomial degree.
\end{remark}

\section{Basic properties of the solution with viscosity}\label{sectionprop}

In this section, we recall some standard properties of the viscous conservation laws~\eqref{vis}, and revisit Kru\v{z}kov's doubling variable (see, for example, \cite{BoucPer1998,Kruzkov1970}) argument to show that the operators $\fe^{\varepsilon^2\Delta}$ in \eqref{eq:continuity} do not ruin the convergence order of the viscous solution.

\subsection{\texorpdfstring{Some properties of \eqref{vis}}{Some properties of the viscous problem}}

\begin{lemma}[Maximal $L^p-L^q$ regularity]\label{lemheat}
If $1<q\le p<\infty$, $s\in\mathbb N_0$, and $0<\varepsilon\le1$, then for the heat equation
$$
u_t-\varepsilon\Delta u=g(u,t,x),\quad u(0,x)=u_0(x),
$$
defined on $\mathbb{T}^d$, we have
$$
\|u_t\|_{L^pW^{s,q}}+\varepsilon\|u\|_{L^pW^{2+s,q}}\lesssim\|g\|_{L^pW^{s,q}}+\varepsilon^{1-\frac1p}\|u_0\|_{W^{s+2-\frac2p,q}}.
$$
\end{lemma}
The proof can be found, for instance, in \cite{Amann1984, Lunardi1989}. The initial trace space is $B^{s+2-2/p}_{q,p}$. For $p\ne2$, the order $r=s+2-2/p$ is noninteger, and
$$
W^{r,q}=B^r_{q,q}\hookrightarrow B^r_{q,p},\qquad q\le p,
$$
where noninteger Sobolev spaces are understood in the Sobolev--Slobodeckij sense. For $p=2$, the required embedding is instead $W^{s+1,q}\hookrightarrow B^{s+1}_{q,2}$, valid for $q\le2$. Thus the initial data may be measured in the stated Sobolev norm.

\begin{lemma}[Product estimate]\label{lem:kp}
Let $m\in\mathbb{N}$, $1<q<\infty$, and let
$1\le r,r_1,r_2,r_3,r_4\le\infty$, $1<q_1,q_2,q_3,q_4\le\infty$ satisfy
$$
\frac1r=\frac1{r_1}+\frac1{r_2}=\frac1{r_3}+\frac1{r_4},
\qquad
\frac1q=\frac1{q_1}+\frac1{q_2}=\frac1{q_3}+\frac1{q_4}.
$$
Then
$$
\begin{aligned}
\|uv\|_{L^r(0,T;W^{m,q})}
\le C\Big(&\|u\|_{L^{r_1}(0,T;W^{m,q_1})}
\|v\|_{L^{r_2}(0,T;L^{q_2})}\\
&+\|u\|_{L^{r_3}(0,T;L^{q_3})}
\|v\|_{L^{r_4}(0,T;W^{m,q_4})}\Big).
\end{aligned}
$$
The constant depends only on $m,d,q,r$ and the exponents $q_i,r_i$, $1\le i\le4$. This is the integer-order form of the Kato--Ponce product estimate; see \cite{BourgLi2014,KatoPonce1988}. For a discussion of the periodic version, see \cite{LiWu2026}.
\end{lemma}

Then, we will recall the famous $L^1$-contraction properties of \eqref{vis}.

\begin{lemma}\label{lemcontra}
Let $u_0,~v_0\in L^\infty\cap\bv(\mathbb{T}^d)$, and let $u^\varepsilon,~v^\varepsilon$ be the unique solution to the viscous conservation laws~\eqref{vis} with initial data $u_0$ and $v_0$, respectively. Then, for any $t\in[0,T]$, we have
$$
\|u^\varepsilon(t)-v^\varepsilon(t)\|_{L^1}\le\|u_0-v_0\|_{L^1}.
$$
\end{lemma}
\begin{proof}
We first set
$$
\gamma_\delta(s)=\sqrt{s^2+\delta^2},
$$
then we have
$$
\gamma'_\delta(s)=\frac{s}{\sqrt{s^2+\delta^2}},
$$
thus we have for any $s\in\mathbb{R}$ that
$$
\lim\limits_{\delta\to0}\gamma_\delta(s)=
\left\{\begin{aligned}
&-s, \quad s\leq0,\\&s, \quad s>0.
\end{aligned}\right.
$$
Moreover, we have for any $s\in\mathbb{R}$ that
$$
\gamma_\delta''(s)=\frac{\delta^2}{(s^2+\delta^2)^\frac32}>0.
$$
Noting that from \eqref{vis} we have
$$
(u^\varepsilon-v^\varepsilon)_t+\nabla\cdot\big(a(t,x)(u^\varepsilon-v^\varepsilon)\big)-\varepsilon\Delta(u^\varepsilon-v^\varepsilon)=0,
$$
where
$$
a(t,x)=\int_0^1f'\big(v^\varepsilon+\theta(u^\varepsilon-v^\varepsilon)\big)\,d\theta,
$$
then $|a|$ is bounded. By testing $\gamma'_\delta(u^\varepsilon-v^\varepsilon)$ on this equation, and denote $w=u^\varepsilon-v^\varepsilon$, we obtain
\begin{equation}\label{dtintrv}
\frac{d}{dt}\int_{\mathbb{T}^d}\gamma_\delta(w)\,dx=-\int_{\mathbb{T}^d}\gamma'_\delta(w)\nabla\cdot(aw)\,dx+\varepsilon\int_{\mathbb{T}^d}\gamma'_\delta(w)\Delta w\,dx.
\end{equation}
By Green's formula, we obtain from \eqref{dtintrv} that
$$
\frac{d}{dt}\int_{\mathbb{T}^d}\gamma_\delta(w)\,dx=\int_{\mathbb{T}^d}\gamma_\delta''(w)wa\cdot\nabla w\,dx-\varepsilon\int_{\mathbb{T}^d}\gamma_\delta''(w)|\nabla w|^2\,dx.
$$
By the inequality
$$
ab\leq\frac\varepsilon2a^2+\frac1{2\varepsilon}b^2
$$
we obtain
\begin{align*}
\frac{d}{dt}\int_{\mathbb{T}^d}\gamma_\delta(w)\,dx&\leq\frac1{2\varepsilon}\sup\limits_{t,x}|a(t,x)|^2\int_{\mathbb{T}^d}\gamma_\delta''(w)w^2\,dx-\frac\varepsilon2\int_{\mathbb{T}^d}\gamma_\delta''(w)|\nabla w|^2\,dx\\
&\le\frac1{2\varepsilon}\sup\limits_{t,x}|a(t,x)|^2\int_{\mathbb{T}^d}\frac{\delta^2w^2}{(w^2+\delta^2)^\frac32}\,dx\lesssim\frac\delta\varepsilon.
\end{align*}
 This implies
$$
\int_{\mathbb{T}^d}\gamma_\delta(w)(t)\,dx\le\|u_0-v_0\|_{L^1}+C\varepsilon^{-1}\delta,
$$
where $C$ is a positive constant that depends on $f$ and $T$.
Note that since
$$
\int_{\mathbb{T}^d}\gamma_\delta(w)(t)\,dx-\|w(t)\|_{L^1}=\int_{\mathbb{T}^d}\Big|\sqrt{w^2(t)+\delta^2}-|w(t)|\Big|\,dx\lesssim\delta,
$$
by taking the limit $\delta\to0$ we obtain
$$
\|w(t)\|_{L^1}=\lim\limits_{\delta\to0}\int_{\mathbb{T}^d}\gamma_\delta(w)(t)\,dx\le\|u_0-v_0\|_{L^1}+\lim\limits_{\delta\to0}C\varepsilon^{-1}\delta=\|u_0-v_0\|_{L^1}.
$$
This concludes the proof.
\end{proof}

We next recall some standard results for the viscous conservation laws~\eqref{vis}.

\begin{lemma}\label{lemmaxpr}
For the viscous conservation laws~\eqref{vis}, we have for any $t\in[0,T]$ that
\begin{align}
\label{rangepr}&
u_-\le u^\varepsilon(t,x)\le u_+\qquad\mbox{a.e. }x\in\mathbb T^d,\\
\label{maxpr}&\|u^\varepsilon(t)\|_{L^\infty}\leq\|u_0\|_{L^\infty},\\
\label{uxl1}&\|u^\varepsilon(t)\|_{W^{1,1}}\leq \|u_0\|_{W^{1,1}},\\
\label{utl1}&\|u^\varepsilon_t(t)\|_{L^1}\le\sup\limits_{u_-\le s\le u_+}|f'(s)|\|u_0\|_{W^{1,1}}+\varepsilon\|u_0\|_{W^{2,1}},
\end{align}
where $u_-$ and $u_+$ are defined in \eqref{u-u+}.
\end{lemma}
For the proof of this lemma, we refer to \cite[Theorem 3]{Kruzkov1970} and \cite[Eq. (11), (14)]{Bardos1979}. 
\begin{remark}\label{rem2}
Taking the limit $\varepsilon=0$, we obtain that the entropy solution of \eqref{periodic} also satisfies the properties listed in Lemmas~\ref{lemcontra} and~\ref{lemmaxpr}. Further, by taking $f\equiv0$, we directly get that the heat semigroup $\fe^{\varepsilon t\Delta}$ also satisfies the properties listed in Lemmas~\ref{lemcontra} and~\ref{lemmaxpr}. Moreover, from the lemma above, we directly obtain
$$
\|\Delta u^\varepsilon\|_{L^\infty L^1}\lesssim\varepsilon^{-1}\|u_0\|_{W^{1,1}}
$$
by the a priori assumption $\|u_0\|_{W^{2,1}}\lesssim\varepsilon^{-1}$.
\end{remark}

\begin{proposition}\label{proputxn}
For the viscous conservation laws~\eqref{vis}, if  $f$ satisfies Assumption~\ref{localsmooth}, and the initial data $u_0$ satisfies
\begin{equation}\label{assume3}
\|u_0\|_{W^{k+1,1}}+\|u_0\|_{W^{k,\infty}}\lesssim_k\varepsilon^{-k}
\end{equation}
for all $k\in\mathbb{N}$, then for any $g\in\mathcal{C}^\infty$ and $k,l\in\mathbb{N}$ with $k+l>0$, we have for all $p\in[1,\infty]$ that
\begin{equation}\label{utxnbd}
\|\partial_t^lg(u^\varepsilon)\|_{L^\infty(0,T;W^{k,p})}\lesssim_{k,l,\delta}\varepsilon^{-k-l+\frac1p-\delta},
\end{equation}
where $\delta>0$ can be taken arbitrarily small.
\end{proposition}
\begin{proof}
By \eqref{assume3}, we have for all $1\le s\in\mathbb{R}$ and $p\in(1,\infty)$ that
\begin{equation}\label{iniest}
\begin{aligned}
\|u_0\|_{W^{s,p}}&\le\|u_0\|^\frac1p_{W^{s,1}}\|u_0\|^{1-\frac1p}_{W^{s,\infty}}\\
&\le\|u_0\|^{\frac1p(m+1-s)}_{W^{m,1}}\|u_0\|^{\frac1p(s-m)}_{W^{m+1,1}}\|u_0\|^{(1-\frac1p)(m+1-s)}_{W^{m,\infty}}\|u_0\|^{(1-\frac1p)(s-m)}_{W^{m+1,\infty}}\lesssim\varepsilon^{-s+\frac1p},
\end{aligned}
\end{equation}
where $m=\lfloor s\rfloor$.

We first prove by induction with respect to $k$ that for any $p>1$ and $g\in\mathcal{C}^\infty$, we have
\begin{equation}\label{lpk}
\|g(u^\varepsilon)\|_{L^p W^{k,p}}\lesssim\varepsilon^{-k}.
\end{equation}
To prove \eqref{lpk}, we first assume without loss of generality that $p>2$ since $L^p\subset L^{p'}$ for all $p>p'$. By \eqref{iniest} and Lemma~\ref{lemheat}, we derive that
\begin{align*}
\|u^\varepsilon\|_{L^p W^{2,p}}&\lesssim\varepsilon^{-1}\|f(u^\varepsilon)\|_{L^p W^{1,p}}+\varepsilon^{-\frac1p}\|u_0\|_{W^{2-\frac2p,p}}\\
&\lesssim\varepsilon^{-1}\sup\limits_{u_-\le u^\varepsilon\le u_+}|f'(u^\varepsilon)|\|u^\varepsilon\|_{L^pW^{1,p}}+\sup\limits_{u_-\le u^\varepsilon\le u_+}|f(u^\varepsilon)|+\varepsilon^{-2+\frac1p}\\
&\lesssim\varepsilon^{-1}\|u^\varepsilon\|^\frac12_{L^pL^p}\|u^\varepsilon\|^\frac12_{L^pW^{2,p}}+\varepsilon^{-2+\frac1p}\lesssim\varepsilon^{-1}\|u^\varepsilon\|^\frac12_{L^pW^{2,p}}+\varepsilon^{-2+\frac1p},
\end{align*}
which implies
$$
\|u^\varepsilon\|_{L^pW^{2,p}}\lesssim\varepsilon^{-2}
$$
and
$$
\|g(u^\varepsilon)\|_{L^pW^{1,p}}\lesssim\sup\limits_{u_-\le u^\varepsilon\le u_+}|g(u^\varepsilon)|+\sup\limits_{u_-\le u^\varepsilon\le u_+}|g'(u^\varepsilon)|\|u^\varepsilon\|_{L^pW^{1,p}}\lesssim\varepsilon^{-1}.
$$
We then assume \eqref{lpk} is valid for all $1\le k\le m$. Then for $k=m+1$, by \eqref{iniest}, {Lemma~\ref{lem:kp}}, and Lemma~\ref{lemheat}, we deduce that
\begin{align*}
\|g&(u^\varepsilon)\|_{L^pW^{m+1,p}}\lesssim\|g'(u^\varepsilon)\nabla u^\varepsilon\|_{L^pW^{m,p}}+\sup\limits_{u_-\le u^\varepsilon\le u_+}|g(u^\varepsilon)|\\
&\lesssim\|g'(u^\varepsilon)\|_{L^\infty L^\infty}\|u^\varepsilon\|_{L^pW^{m+1,p}}+\|g'(u^\varepsilon)\|_{L^{2p}W^{m,2p}}\|u^\varepsilon\|_{L^{2p}W^{1,2p}}+1\\
&\lesssim\varepsilon^{-1}\|f(u^\varepsilon)\|_{L^p W^{m,p}}+\varepsilon^{-\frac1p}\|u_0\|_{W^{m+1-\frac2p,p}}+\varepsilon^{-m-1}\lesssim\varepsilon^{-m-1}.
\end{align*}
We then conclude \eqref{lpk}. Thus by Young's inequality, we obtain for any $\delta>0$ sufficiently small that
\begin{equation}\label{w1infk}
\|g(u^\varepsilon)\|_{L^pW^{k,\infty}}\lesssim\|g(u^\varepsilon)\|^{1-\delta/2}_{L^pW^{k,p}}\|g(u^\varepsilon)\|^{\delta/2}_{L^pW^{k+1,p}}\lesssim\varepsilon^{-k-\delta/2}
\end{equation}
by taking $p>\frac {2d}\delta$.

In the following spatial and temporal inductions, the loss parameters in the induction hypotheses are chosen sufficiently small so that the total loss is at most $\delta/2$. We now prove
\begin{equation}\label{uxnbd}
\|g(u^\varepsilon)\|_{L^p([0,T];W^{k,q})}\lesssim_{k,l,\delta}\varepsilon^{-k+\frac1q-\frac\delta2}   
\end{equation}
for any $p\ge q>1$ and $\delta>0$ by induction with respect to $k$. Note that the case $k=1$ can be obtained directly by interpolating between \eqref{uxl1} and \eqref{w1infk}. Assume \eqref{uxnbd} is valid for $l=0$, and all $p\ge q>1$, $1\le k\le m$. Then for $k=m+1$, again by \eqref{iniest}, \eqref{lpk}, Lemma~\ref{lem:kp}, and Lemma~\ref{lemheat}, we deduce that
\begin{align*}
\|g&(u^\varepsilon)\|_{L^pW^{m+1,q}}\lesssim\|g'(u^\varepsilon)\nabla u^\varepsilon\|_{L^pW^{m,q}}+\sup\limits_{u_-\le u^\varepsilon\le u_+}|g(u^\varepsilon)|\\
&\lesssim\sup\limits_{u_-\le u^\varepsilon\le u_+}|g'(u^\varepsilon)|\|u^\varepsilon\|_{L^pW^{m+1,q}}+\|g'(u^\varepsilon)\|_{L^{p_1}W^{m,q_1}}\|u^\varepsilon\|_{L^{p_2}W^{1,p_2}}+1\\
&\lesssim\varepsilon^{-1}\|f(u^\varepsilon)\|_{L^pW^{m,q}}+\varepsilon^{-\frac1p}\|u_0\|_{W^{m+1-\frac2p,q}}+\varepsilon^{-m+\frac1q-\frac\delta2-1}\lesssim\varepsilon^{-m-1+\frac1q-\frac\delta2},
\end{align*}
where
\begin{equation}\label{p1q1}
p_1=\frac{2p}{2-p\delta_1},\qquad q_1=\frac{2q}{2-q\delta_1},\qquad p_2=\frac2{\delta_1}
\end{equation}
with $\delta_1=\min\{\frac1p,\frac\delta2\}$ and thus $\frac1{q_1}>\frac1q-\frac\delta2$. Moreover, since
\begin{equation}\label{q=1}
\|g(u^\varepsilon)\|_{L^pW^{k,1}}\lesssim\|g(u^\varepsilon)\|_{L^pW^{k,1+\frac\delta4}}
\lesssim\varepsilon^{-k+\frac1{1+\frac\delta4}-\frac\delta4}\lesssim\varepsilon^{-k+1-\frac\delta2}, 
\end{equation}
We thus conclude \eqref{uxnbd}.

We now step to the proof for
\begin{equation}\label{utxnpq}
\|\partial_t^lg(u^\varepsilon)\|_{L^p([0,T];W^{k,q})}\lesssim_{k,l,\delta}\varepsilon^{-k-l+\frac1q-\frac\delta2}
\end{equation}
by induction for $l$. We assume \eqref{utxnpq} is valid for all $0\le l\le m$. Then for $l=m+1$ and any $p\ge q>1$, we obtain by \eqref{vis}, \eqref{iniest}, and {Lemma~\ref{lem:kp}} that
\begin{align*}
\|&\partial_t^{m+1}g(u^\varepsilon)\|_{L^p W^{k,q}}\lesssim\|\partial_t^m\big(g'(u^\varepsilon)u^\varepsilon_t\big)\|_{L^pW^{k,q}}\\
&\lesssim\sup\limits_{\substack{0\le j\le m\\k\ge1}}\|\partial_t^jg'(u^\varepsilon)\|_{L^{p_1}W^{k,q_1}}\|\partial_t^{m+1-j}u^\varepsilon\|_{L^{p_2}L^{p_2}}+\sup\limits_{0\le j\le m}\|\partial_t^jg'(u^\varepsilon)\|_{L^{p_2}L^{p_2}}\|\partial_t^{m+1-j}u^\varepsilon\|_{L^{p_1}W^{k,q_1}}\\
&\lesssim\sup\limits_{\substack{0\le j\le m\\k\ge1}}\Big(\varepsilon^{-k-j+\frac1q-\frac\delta2}\|\partial_t^{m-j}f(u^\varepsilon)\|_{L^{p_2} W^{1,p_2}}+\varepsilon^{1-k-j+\frac1q-\frac\delta2}\|\partial_t^{m-j}u^\varepsilon\|_{L^{p_2} W^{2,p_2}}\big)\\
&\quad+\sup\limits_{0\le j\le m}\Big(\varepsilon^{-j+\frac1{p_2}-\frac{\delta'}2}\|\partial_t^{m-j}f(u^\varepsilon)\|_{L^{p_1}W^{k+1,q_1}}+\varepsilon^{1-j+\frac1{p_2}-\frac{\delta'}2}\|\partial_t^{m-j}u^\varepsilon\|_{L^{p_1}W^{k+2,q_1}}\big)\\
&\lesssim\sup\limits_{0\le j\le m}\varepsilon^{-k-j-1+\frac1q-\frac\delta2}\varepsilon^{-m+j+\frac1{p_2}-\frac{\delta'}2}+\sup\limits_{0\le j\le m}\varepsilon^{-j+\frac1{p_2}-\frac{\delta'}2}\varepsilon^{-k-m+j-1+\frac1q-\frac\delta2}=\varepsilon^{-k-m-1+\frac1q-\frac\delta2},
\end{align*}
where $p_1,~q_1,~p_2$ are defined in \eqref{p1q1}, and $\delta'=\frac2{p_2}$. Since the borderline cases $q=1,\infty$ can also be obtained similarly as given in \eqref{q=1} and \eqref{w1infk}, respectively, we arrive at \eqref{utxnpq} for general $k,l\in\mathbb{N}$ with $k+l>0$. Finally, by taking $p>\frac2\delta$, we obtain by Sobolev embedding that
\begin{align*}
\|\partial_t^lg(u^\varepsilon)\|_{L^\infty([0,T];W^{k,q})}&\lesssim\|\partial_t^lg(u^\varepsilon)\|_{L^p([0,T];W^{k,q})}+\|\partial_t^lg(u^\varepsilon)\|^{1-\frac\delta2}_{L^p([0,T];W^{k,q})} \|\partial_t^{l+1}g(u^\varepsilon)\|^{\frac\delta2}_{L^p([0,T];W^{k,q})}\\
&\lesssim_{k,l,\delta}\varepsilon^{-k-l+\frac1q-\delta},
\end{align*}
this concludes the proof.
\end{proof}

\subsection{The convergence order of a splitting flow}

In this part, we study the effect of the operators $\fe^{\varepsilon^2\Delta}$ at every time step. We recall that we added these operators in order to fit the requirements of Proposition~\ref{proputxn} and avoid the residual of the optimizer losing derivatives continuously. However, these operators also affect the exact solution we simulate. We show that, by taking $\tau\sim\varepsilon$, such effect is limited, which is parallel to the standard result \eqref{visconv}. Specifically, we will prove the following proposition:

\begin{proposition}\label{prop:splitting-flow}
Let $t_k=k\tau$ with $\tau=c\varepsilon$. Set $v^\varepsilon$ to be a piecewise continuous (in time) function such that, $v^\varepsilon\big|_{(t_k,t_{k+1}]}$ is the exact solution to the viscous equation \eqref{vis}, where the initial data is defined by 
$$
v^\varepsilon(0)=u_0,\qquad
(v^\varepsilon)^+(t_k)=\lim\limits_{t\to t_k^+}v^\varepsilon(t)=\fe^{\varepsilon^2\Delta}v^\varepsilon(t_k),\quad k\ge0.
$$
Here $v^\varepsilon(t_k)$ denotes the left trace for $k\ge1$, while $v^\varepsilon(0)=u_0$ is the data before the first heat step. Thus $(v^\varepsilon)^+(0)=\fe^{\varepsilon^2\Delta}u_0$. Let $u$ be the entropy solution with initial data $u_0$. Then, we have the following error estimate: for any $T>0$, $u_0\in L^\infty\cap\bv$, and $f\in\mathcal{C}^\infty$, there exists $\varepsilon_0>0$, such that for any  $t\in(0,T]$, $\varepsilon\in(0,\varepsilon_0]$, we have
$$
\|u(t)-v^\varepsilon(t)\|_{L^1}\le C\varepsilon^\frac12,
$$
where $C$ depends on $f$, $T$, $c$ and $u_0$, but is independent of $t$ and $\varepsilon$.
\end{proposition}
\begin{proof}
For simplicity of notation in this proof, we write $v=v^\varepsilon$. We will revisit the doubling variable argument by Kru\v{z}kov \cite{Kruzkov1970}. Note that if we define the nonlinear operator $S^\varepsilon_\tau(w)$ as the exact solution of \eqref{vis} at time $t=\tau$ with initial data $w$, then $v(t_n)$ is the splitting flow:
$$
v(t_n)=S_\tau^\varepsilon\big(\fe^{\varepsilon^2\Delta}v(t_{n-1})\big)=\big(S_\tau^\varepsilon\circ\fe^{\varepsilon^2\Delta}\big)^nu_0.
$$
We will follow the proof for operator splitting given in \cite[Chap. V]{Holdenetal2010}.

Let $\eta\in\mathcal{C}_c^\infty(\mathbb{R})$ and $\zeta\in\mathcal{C}_c^\infty(\mathbb{R}^d)$ be non-negative, even functions compactly supported in $[-2,2]$ and $[-2,2]^d$, respectively, such that
$$
\int_{-2}^2 \eta(s)\,ds = \int_{[-2,2]^d} \zeta(x)\,dx = 1.
$$
For any $\theta,\delta>0$, we denote the rescaled profiles by $\eta_\theta(s)=\theta^{-1}\eta(s/\theta)$ and $\zeta_\delta(x)=\delta^{-d}\zeta(x/\delta)$. For $2\delta<\pi$, we view the spatial kernel as a smooth
function on $\mathbb T^d$, using its support inside $(-\pi,\pi)^d$; $x-y$ is understood modulo the period. By a simple change of variables, it straightforwardly follows that
\begin{equation}\label{test+1}
\int_{-2\theta}^{2\theta}|s|\eta_\theta(s)\,ds \lesssim\theta, \qquad 
\int_{[-2\delta,2\delta]^d}|x|\zeta_\delta(x)\,dx \lesssim\delta
\end{equation}
and
\begin{equation}\label{test-1}
\|\eta_\theta'\|_{L^1} \lesssim\theta^{-1}, \qquad
\|\nabla\zeta_\delta\|_{L^1} \lesssim\delta^{-1}.
\end{equation}
For fixed $(s,y)\in(0,T)\times\mathbb{T}^d$, we have Kru\v{z}kov's entropy condition:
\begin{equation}\label{entu}
\partial_r|u(r,x)-v(s,y)|+\nabla_x\cdot\Big(\sgn\big(u(r,x)-v(s,y)\big)\big(f(u)-f(v)\big)\Big)\le0.
\end{equation}
Inequality \eqref{entu} holds in the sense of distributions in $(r,x)$, with the fixed value $v(s,y)$ serving as the entropy parameter; equivalently, it is tested against nonnegative smooth test functions. Similarly, for the viscous solution, we also have the estimate \cite[Eq. (4.23)]{Kruzkov1970}
\begin{equation}\label{entv}
\partial_s|u(r,x)-v(s,y)|+\nabla_y\cdot\Big(\sgn\big(u(r,x)-v(s,y)\big)\big(f(u)-f(v)\big)\Big)\le\varepsilon\Delta_y|u(r,x)-v(s,y)|
\end{equation}
since $|v-u|=|u-v|$ and $\sgn(u-v)=-\sgn(v-u)$. Inequality \eqref{entv} holds in the distributional sense on each open time slab. We integrate by parts in time slab by slab, retaining the jump terms in \eqref{intbpt}. By applying \eqref{entu} on $(0,t)$ and \eqref{entv} on each time slab $(s_{n-1},s_n)$, we first introduce auxiliary temporal cutoffs compactly supported in the corresponding open intervals and test with $\eta_\theta(r-s)\zeta_\delta(x-y)$. Letting the auxiliary cutoffs tend to one and summing over all time slabs, we obtain
\begin{equation}\label{splitst}
\begin{aligned}
\iiiint_{[0,t]^2\times\mathbb{T}^{2d}}&\Big((\partial_r+\partial_s)\big|u(r,x)-v(s,y)\big|\Big)\eta_\theta(r-s)\zeta_\delta(x-y)\\
&+\Big[(\nabla_x+\nabla_y)\cdot\Big(\sgn\big(u(r,x)-v(s,y)\big)\big(f(u)-f(v)\big)\Big)\Big]\eta_\theta(r-s)\zeta_\delta(x-y)\\
&-\varepsilon\Big(\Delta_y\big|u(r,x)-v(s,y)\big|\Big)\eta_\theta(r-s)\zeta_\delta(x-y)\,dxdydrds\le0.
\end{aligned}
\end{equation} 
Integration by parts in space, noting that $\nabla_x\zeta_\delta(x-y)=-\nabla_y\zeta_\delta(x-y)$, we directly get
$$
\iiiint_{[0,t]^2\times\mathbb{T}^{2d}}\Big[(\nabla_x+\nabla_y)\cdot\Big(\sgn\big(u(r,x)-v(s,y)\big)\big(f(u)-f(v)\big)\Big)\Big]\eta_\theta(r-s)\zeta_\delta(x-y)\,dxdydrds=0.
$$
Moreover, by \eqref{uxl1} and Remark~\ref{rem2}, and using again integration by parts, we deduce that
\begin{align*}
\Big|\iiiint_{[0,t]^2\times\mathbb{T}^{2d}}&\Big(\Delta_y\big|u(r,x)-v(s,y)\big|\Big)\eta_\theta(r-s)\zeta_\delta(x-y)\,dxdydrds\Big|\\
&\le\iiiint_{[0,t]^2\times\mathbb{T}^{2d}}|\sgn(u-v)|\big|\nabla_yv(s,y)\big|\eta_\theta(r-s)\big|\nabla_y\zeta_\delta(x-y)\big|\,dxdydrds\\
&\le\iiint_{[0,t]^2\times\mathbb{T}^d}\eta_\theta(r-s)\big|\nabla v(s)\big|*_x\big|\nabla\zeta_\delta\big|\,dxdrds
\end{align*}
Note that for any fixed $s\in[0,T]$ we have
$$
\int_0^t\eta_\theta(r-s)\,dr\le1,
$$
we thus obtain by Young's inequality for convolutions, \eqref{test-1} and \eqref{uxl1} that
\begin{align*}
\Big|\iiiint_{[0,t]^2\times\mathbb{T}^{2d}}&\Big(\Delta_y\big|u(r,x)-v(s,y)\big|\Big)\eta_\theta(r-s)\zeta_\delta(x-y)\,dxdydrds\Big|\\
&\le\int_0^t\tv_y(v(s))\tv(\zeta_\delta)\,ds\lesssim\delta^{-1}.
\end{align*}
This implies
\begin{equation}\label{splitres}
\iiiint_{[0,t]^2\times\mathbb{T}^{2d}}\big((\partial_r+\partial_s)|u(r,x)-v(s,y)|\big)\eta_\theta(r-s)\zeta_\delta(x-y)\,dxdydrds\lesssim\frac\varepsilon\delta.
\end{equation}
Noting that $\partial_r\eta_\theta(r-s)=-\partial_s\eta_\theta(r-s)$, using integration by parts, we derive that
\begin{equation}\label{intbpt}
\begin{aligned}
&\iiiint_{[0,t]^2\times\mathbb{T}^{2d}}\Big((\partial_r+\partial_s)\big|u(r,x)-v(s,y)\big|\Big)\eta_\theta(r-s)\zeta_\delta(x-y)\,dxdydrds\\
&=\iiint_{[0,t]\times\mathbb{T}^{2d}}\Big(\big|u(t,x)-v(s,y)\big|\eta_\theta(t-s)-\big|u(0,x)-v(s,y)\big|\eta_\theta(-s)\Big)\zeta_\delta(x-y)\,dxdyds\\
&\quad+\sum\limits_{n=1}^m\iiint_{[0,t]\times\mathbb{T}^{2d}}\big|u(r,x)-v(s_n,y)\big|\eta_\theta(r-s_n)\zeta_\delta(x-y)\,dxdydr\\
&\quad-\sum\limits_{n=0}^{m-1}\iiint_{[0,t]\times\mathbb{T}^{2d}}\big|u(r,x)-\fe^{\varepsilon^2\Delta_y}v(s_n,y)\big|\eta_\theta(r-s_n)\zeta_\delta(x-y)\,dxdydr
\end{aligned}
\end{equation}
where $m=\lceil\frac t\tau\rceil,~s_n=\min\{n\tau,t\}$.

By denoting $G$ the heat kernel satisfying $\fe^{\varepsilon^2\Delta}v(s_n)=G*_yv(s_n)$, we have
\begin{align*}
|u(r,x)&-\fe^{\varepsilon^2\Delta_y}v(s_n,y)\big|=\Big|\int_{\mathbb{T}^d}G(y-z)\big(u(r,x)-v(s_n,z)\big)\,dz\Big|\\
&\le\int_{\mathbb{T}^d}G(y-z)\big|u(r,x)-v(s_n,z)\big|\,dz=\fe^{\varepsilon^2\Delta_y}\big|u(r,x)-v(s_n,y)\big|,
\end{align*}
which implies for any $1\le n\le m-1$ that
\begin{equation}\label{splitest}
\begin{aligned}
&\iiint_{[0,t]\times\mathbb{T}^{2d}}\Big(\big|u(r,x)-\fe^{\varepsilon^2\Delta_y}v(s_n,y)\big|-\big|u(r,x)-v(s_n,y)\big|\Big)\eta_\theta(r-s_n)\zeta_\delta(x-y)\,dxdydr\\
&\le\iiint_{[0,t]\times\mathbb{T}^{2d}}(\fe^{\varepsilon^2\Delta_y}-I)\big|u(r,x)-v(s_n,y)\big|\eta_\theta(r-s_n)\zeta_\delta(x-y)\,dxdydr\\
&\le\varepsilon^2\sup_{p\in[0,\varepsilon^2]}\iiint_{[0,t]\times\mathbb{T}^{2d}}\left|\fe^{p\Delta_y}\nabla_y\big|u(r,x)-v(s_n,y)\big|\cdot\nabla_y\zeta_\delta(x-y)\right|\eta_\theta(r-s_n)\,dxdydr\\
&\le\varepsilon^2\iiint_{[0,t]\times\mathbb{T}^{2d}}\big|\nabla_yv(s_n,y)\big|\cdot\big|\nabla_x\zeta_\delta(x-y)\big|\eta_\theta(r-s_n)\,dxdydr\lesssim\frac{\varepsilon^2}\delta.
\end{aligned}
\end{equation}
Here $v(s_n)=v(s_n^-)$ and
$\fe^{\varepsilon^2\Delta}v(s_n)=v(s_n^+)$, so \eqref{splitest}
estimates the heat-step jump.

For the endpoint estimates, the BV translation bound and heat smoothing give
$$
\|u_0(\cdot+z)-u_0\|_{L^1}\le |z|\tv(u_0),\qquad \|(I-\fe^{\varepsilon^2\Delta})u_0\|_{L^1} \le C\varepsilon\tv(u_0).
$$
For $a_n=\fe^{\varepsilon^2\Delta}v(t_n)$, heat smoothing gives $\|\Delta a_n\|_{L^1}\le C\varepsilon^{-1}\tv(v(t_n))$. Since both evolution operators decrease total variation, \eqref{utl1} implies
$$
\sup_{t_n<s<t_{n+1}}\|v_s(s)\|_{L^1} \le\sup_{[u_-,u_+]}|f'|\tv(a_n) +\varepsilon\|\Delta a_n\|_{L^1} \le C\tv(u_0).
$$
Hence $\|v(s)-v(r)\|_{L^1}\le C|s-r|$ within each slab, with the corresponding one-sided endpoint values. The conservation law and the BV bound for $u$ also give $\|u(s)-u(r)\|_{L^1}\le\sup_{[u_-,u_+]}|f'|\tv(u_0)|s-r|$. These constants are independent of the slab index and $\varepsilon$.

Choose $\theta$ such that $0<2\theta<t-t_{m-1}\le\tau$. On the support of $\eta_\theta(s)$, since $0\le s\le2\theta<\tau$, the first-slab estimate with $v^+(0)=\fe^{\varepsilon^2\Delta}u_0$ yields
$$
\|v(s)-v^+(0)\|_{L^1} \le\int_0^s\|v_r(r)\|_{L^1}\,dr \le s\sup_{0<r<\tau}\|v_r(r)\|_{L^1}\le Cs.
$$

Using $u(0)=u_0$ and $v^+(0)=\fe^{\varepsilon^2\Delta}u_0$, we insert $u_0(y)$ and $v^+(0,y)$ and apply the triangle inequality. Together with these estimates, \eqref{test+1}, \eqref{test-1}, and $\theta\lesssim\tau\sim\varepsilon$, we obtain
\begin{align}\label{utv0}
&\quad\iiint_{[0,t]\times\mathbb{T}^{2d}}\big|u(0,x)-v(s,y)\big|\eta_\theta(-s)\zeta_\delta(x-y)\,dxdyds\notag\\
&\le\iiint_{[0,t]\times\mathbb{T}^{2d}}\Big(\big|u_0(x)-u_0(y)\big|+\big|(I-\fe^{\varepsilon^2\Delta_y})u_0(y)\big|\notag\\*
&\qquad\qquad+\big|v^+(0,y)-v(s,y)\big|\Big)\eta_\theta(s)\zeta_\delta(x-y)\,dxdyds\notag\\
&\le\frac12\iint_{\mathbb{T}^{2d}}\big|u_0(x)-u_0(x+z)\big|\zeta_\delta(z)\,dxdz+\frac12\iint_{\mathbb{T}^{2d}}\big|(I-\fe^{\varepsilon^2\Delta_y})u_0(y)\big|\zeta_\delta(x-y)\,dxdy\notag\\*
&\qquad\qquad+\iiint_{[0,t]\times\mathbb{T}^{2d}}\big|v^+(0,y)-v(s,y)\big|\eta_\theta(s)\zeta_\delta(x-y)\,dxdyds\notag\\
&\lesssim\tv(u_0)\int_{\mathbb T^d}|z|\zeta_\delta(z)\,dz+\varepsilon\tv(u_0)+\int_0^{2\theta}s\eta_\theta(s)\,ds\;\sup_{0<r<\tau}\|v_r(r)\|_{L^1}\notag\\
&\lesssim\delta\tv(u_0)+\varepsilon\tv(u_0)+\theta\tv(u_0)\lesssim\delta+\varepsilon.
\end{align}
Similarly, we also deduce that
\begin{equation}\label{u0vs}
\iiint_{[0,t]\times\mathbb{T}^{2d}}\big|u(r,x)-v^+(0,y)\big|\eta_\theta(r)\zeta_\delta(x-y)\,dxdydr\lesssim\delta+\varepsilon.
\end{equation}
Combining the estimates between \eqref{splitres} and \eqref{u0vs}, noting that $m-1\le\frac t\tau\le\frac T\tau$, we arrive at
\begin{equation}\label{uvT}
\begin{aligned}
\iiint_{[0,t]\times\mathbb{T}^{2d}}&\Big(\big|u(t,x)-v(s,y)\big|+\big|u(s,x)-v(t,y)\big|\Big)\eta_\theta(t-s)\zeta_\delta(x-y)\,dxdyds\\
&=\iiint_{[0,t]\times\mathbb{T}^{2d}}\big|u(t,x)-v(s,y)\big|\eta_\theta(t-s)\zeta_\delta(x-y)\,dxdyds\\
&\quad+\iiint_{[0,t]\times\mathbb{T}^{2d}}\big|u(s,x)-v(t,y)\big|\eta_\theta(s-t)\zeta_\delta(x-y)\,dxdyds\lesssim\delta+\varepsilon+\frac{\varepsilon}{\delta}.
\end{aligned}
\end{equation}
By the definition of $\eta$ and $\zeta$, we directly have
$$
\iiint_{[0,t]\times\mathbb{T}^{2d}}\big|u(t,y)-v(t,y)\big|\eta_\theta(t-s)\zeta_\delta(x-y)\,dxdyds=\frac12\|u(t)-v(t)\|_{L^1},
$$
which implies
\begin{align*}
\|u(t)&-v(t)\|_{L^1}\le\iiint_{[0,t]\times\mathbb{T}^{2d}}\Big(\big|u(t,x)-u(t,y)\big|+\big|v(t,y)-v(s,y)\big|+\big|u(s,x)-u(s,y)\big|\\
&\quad+\big|u(s,y)-u(t,y)\big|\Big)\eta_\theta(t-s)\zeta_\delta(x-y)\,dxdyds+C\big(\delta+\varepsilon+\frac{\varepsilon}{\delta}\big),
\end{align*}
where $C$ is a generic constant independent of $\theta,\delta$ and $\varepsilon$. The spatial terms are bounded by $C\delta$ using the BV translation estimate. Since $2\theta<t-t_{m-1}$, $s>t_{m-1}$, the times $s$ and $t$ on the support of $\eta_\theta(t-s)$ lie in the same final slab. Hence the temporal terms are bounded by $C\theta\lesssim\varepsilon$. We finally derive that
$$
\|u(t)-v(t)\|_{L^1}\lesssim \delta+\varepsilon+\frac{\varepsilon}{\delta}.
$$
We conclude the proposition by taking $\delta=\sqrt{\varepsilon}$.
\end{proof}

\section{A space–time error estimate for the minimized residual}\label{sectionerroranal}

In this section we prove Theorem~\ref{thm}. We first compare the optimizer with exact viscous solutions whose left endpoint data agree with the numerical heat-smoothed data, and then compare these auxiliary solutions with the exact splitting flow.  Further, we assume $f$ satisfies Assumption~\ref{localsmooth} throughout this section.

\subsection{Reference and auxiliary flows}

Let $v^\varepsilon$ be the exact splitting flow defined in Proposition~\ref{prop:splitting-flow}.  Equivalently, we may write this splitting flow with the cut-off flux $\widetilde f$, since the heat semigroup and the viscous maximum principle keep $v^\varepsilon$ in the interval where $\widetilde f=f$.  On each slab $I_i=(T_i,T_{i+1}]$, write $v_i^\varepsilon=v^\varepsilon|_{I_i}$, where
\begin{equation}\label{eq:v-split}
\partial_t v_i^\varepsilon+\nabla\cdot {\widetilde f(v_i^\varepsilon)}
-\varepsilon\Delta v_i^\varepsilon=0,\qquad t\in I_i,
\end{equation}
and
\begin{equation}\label{eq:v-split-cont}
v_0^\varepsilon(T_0)=\fe^{\varepsilon^2\Delta}u_0,\qquad
v_i^\varepsilon(T_i)=\fe^{\varepsilon^2\Delta}v_{i-1}^\varepsilon(T_i),
\quad i\ge1 .
\end{equation}
Here $v_0^\varepsilon(T_0)$ is the right limit after the initial heat operator is applied. Therefore, under $\tau\sim\varepsilon$,
\begin{equation}\label{eq:v-u-error}
\|v^\varepsilon(t)-u(t)\|_{L^1(\mathbb{T}^d)} \le C\varepsilon^{1/2},\qquad 0\le t\le T .
\end{equation}
It remains to estimate $u_N^\varepsilon-v^\varepsilon$.

For the numerical residual, set
\begin{equation}\label{eq:gNi}
g_{N,i}:=\partial_tu_{N,i}^\varepsilon+\nabla\cdot\widetilde f(u_{N,i}^\varepsilon)-\varepsilon\Delta u_{N,i}^\varepsilon, \qquad t\in I_i .
\end{equation}

For the proof, define the auxiliary exact solution $\widetilde v_i^\varepsilon$ on each slab by
\begin{equation}\label{eq:tilde-v}
\partial_t\widetilde v_i^\varepsilon+\nabla\cdot\widetilde f(\widetilde v_i^\varepsilon)-\varepsilon\Delta\widetilde v_i^\varepsilon=0,\qquad t\in I_i,
\end{equation}
with the same left endpoint as the numerical solution:
\begin{equation}\label{eq:tilde-v-data}
\widetilde v_0^\varepsilon(T_0)=P_N\fe^{\varepsilon^2\Delta}u_0,\qquad
\widetilde v_i^\varepsilon(T_i)=\fe^{\varepsilon^2\Delta}u_{N,i-1}^\varepsilon(T_i),\quad i\ge1 .
\end{equation}

\subsection{Local residual and one-step error}

We first record the temporal interpolation estimates used below. For a Banach space $X$, $w\in C^{m+1}(\overline I_i;X)$, and $1\le m\le k-1$,
\begin{align*}
\|(I-I_k)w\|_{L^\infty(I_i;X)}&\le C_m\ell_k\tau^m k^{-m}\|\partial_t^m w\|_{L^\infty(I_i;X)},\\
\|(\partial_t I_k-I_k\partial_t)w\|_{L^\infty(I_i;X)}&\le C_m\ell_k\tau^m k^{-m}\|\partial_t^{m+1}w\|_{L^\infty(I_i;X)}.
\end{align*}
Here $C_m$ is independent of $k$ and $\tau$. These bounds follow by rescaling from Jackson's inequality \cite[Theorem 3.8]{RyabenkiiTsynkov2006}, the $O(\ell_k)$ Lebesgue constant, and Forst's derivative interpolation estimate \cite{Forst1983}, using $(\partial_t I_k-I_k\partial_t)w=\partial_t(I_kw-w)-(I_k-I)\partial_t w$. The $X$-valued bounds follow by duality.

We now prove the residual estimate on a single slab.

\begin{lemma}\label{cor:comparison-admissible}
The comparison function $I_kP_N\widetilde v_i^\varepsilon$ belongs to $\widehat{\mathcal A}_{i,k,N}$.
\end{lemma}
\begin{proof}
The interpolation nodes include the left endpoint $T_i$, and the endpoint data $\widetilde v_i^\varepsilon(T_i)$ belong to the Fourier space $X$. Therefore
$$
I_kP_N\widetilde v_i^\varepsilon(T_i) =P_N\widetilde v_i^\varepsilon(T_i) =\widetilde v_i^\varepsilon(T_i)=w_i,
$$
which is exactly the defining condition of $\widehat{\mathcal A}_{i,k,N}$.
\end{proof}

\begin{lemma}[Comparison with the reference flow]\label{lem:reference-comparison}
Set
$$
E_0:=\|\widetilde v_0^\varepsilon(T_0)-v_0^\varepsilon(T_0)\|_{L^1},\qquad
E_i:=\|u_{N,i-1}^\varepsilon(T_i)-v_{i-1}^\varepsilon(T_i)\|_{L^1}\quad(i\ge1).
$$
For every integer $m\ge1$,
\begin{equation}\label{eq:tilde-v0}
E_0\le C_m N^{-m}\varepsilon^{-m},
\end{equation}
and, on each slab $I_i$,
\begin{equation}\label{eq:tilde-v-recursion}
\|\widetilde v_i^\varepsilon-v_i^\varepsilon\|_{L^\infty(I_i;L^1)}\le E_i.
\end{equation}
Here $C_m$ is independent of $N$ and $\varepsilon$.
\end{lemma}
\begin{proof}
The initial data differ only by the spatial projection. Using the
$L^2$ spectral estimate and heat smoothing gives
$$
E_0=\|(P_N-I)\fe^{\varepsilon^2\Delta}u_0\|_{L^1} \lesssim\|(P_N-I)\fe^{\varepsilon^2\Delta}u_0\|_{L^2}\lesssim N^{-m}\|\fe^{\varepsilon^2\Delta}u_0\|_{H^m}\lesssim N^{-m}\varepsilon^{-m}\|u_0\|_{L^2}.
$$
This avoids endpoint $L^1$-boundedness of $P_N$. For $i\ge1$, \eqref{eq:tilde-v-data}, \eqref{eq:v-split-cont}, and heat-semigroup contraction imply
$$
\|\widetilde v_i^\varepsilon(T_i)-v_i^\varepsilon(T_i)\|_{L^1}=\|\fe^{\varepsilon^2\Delta}
(u_{N,i-1}^\varepsilon(T_i)-v_{i-1}^\varepsilon(T_i))\|_{L^1}\le E_i.
$$
The $L^1$-contraction of the viscous flow (Lemma~\ref{lemcontra}) then proves \eqref{eq:tilde-v-recursion} for every $i\ge0$.
\end{proof}

\begin{lemma}\label{lem:local-residual-opti}
Assume that $\widetilde v_i^\varepsilon(T_i)$ satisfies the hypotheses of Proposition~\ref{proputxn}. Define
\begin{equation}\label{eq:ri}
R_i:=\partial_tI_kP_N\widetilde v_i^\varepsilon+\nabla\cdot {\widetilde f(I_kP_N\widetilde v_i^\varepsilon)}-\varepsilon\Delta I_kP_N\widetilde v_i^\varepsilon.
\end{equation}
Then
\begin{equation}\label{eq:ri-bound}
\|R_i\|_{L^1(I_i\times\mathbb{T}^d)}\le C\tau\bigl(\varepsilon^{-m-3\delta}\tau^m \ell_kk^{-m}
+\varepsilon^{-m-3\delta}N^{-m}\bigr).
\end{equation}
Moreover, the numerical residual satisfies
\begin{equation}\label{eq:g-bound-opti}
\|g_{N,i}\|_{L^1(I_i\times\mathbb{T}^d)}\le\tau E_i+C\tau\bigl(\varepsilon^{-m-3\delta}\tau^m \ell_kk^{-m}+\varepsilon^{-m-3\delta}N^{-m}\bigr).
\end{equation}
\end{lemma}

\begin{proof}
Applying $I_kP_N$ to \eqref{eq:tilde-v}, and using that $P_N$ and $\Delta$ commute, gives
$$
I_kP_N\partial_t\widetilde v_i^\varepsilon +I_kP_N\nabla\cdot {\widetilde f(\widetilde v_i^\varepsilon)}-\varepsilon\Delta(I_kP_N\widetilde v_i^\varepsilon)=0.
$$
Here we do not commute $I_k$ with $\partial_t$. Therefore the residual is
\begin{align*}
R_i&=\nabla\cdot {\widetilde f(I_kP_N\widetilde v_i^\varepsilon)}-I_kP_N\nabla\cdot {\widetilde f(\widetilde v_i^\varepsilon)}+(\partial_tI_k-I_k\partial_t)P_N\widetilde v_i^\varepsilon .
\end{align*}
We split
\begin{align}\label{eq:R-split}
\|R_i\|_{L^1(I_i\times\mathbb{T}^d)}&\le\tau\|\nabla\cdot({\widetilde f}(I_kP_N\widetilde v_i^\varepsilon)-{\widetilde f}(\widetilde v_i^\varepsilon))\|_{L_t^\infty L_x^{1+\delta}}\notag\\
&\quad+\tau\|(I-I_kP_N)\nabla\cdot {\widetilde f}(\widetilde v_i^\varepsilon)\|_{L_t^\infty L_x^{1+\delta}}\notag\\
&\quad+\tau\|(\partial_tI_k-I_k\partial_t)P_N\widetilde v_i^\varepsilon\|_{L_t^\infty L_x^{1+\delta}}.
\end{align}
All norms on the right are taken over $I_i\times\mathbb{T}^d$.

For the first term in \eqref{eq:R-split}, write
\begin{equation}\label{defzi}
Z_i:=I_kP_N\widetilde v_i^\varepsilon-\widetilde v_i^\varepsilon=(I_k-I)P_N\widetilde v_i^\varepsilon+(P_N-I)\widetilde v_i^\varepsilon .
\end{equation}
For each component $\widetilde f_\ell$ of $\widetilde f$,
$$
\partial_{x_\ell}[\widetilde f_\ell(I_kP_N\widetilde v_i^\varepsilon)-\widetilde f_\ell(\widetilde v_i^\varepsilon)]=\widetilde f_\ell'(I_kP_N\widetilde v_i^\varepsilon)\partial_{x_\ell}Z_i+\bigl(\widetilde f_\ell'(I_kP_N\widetilde v_i^\varepsilon)-\widetilde f_\ell'(\widetilde v_i^\varepsilon)\bigr)\partial_{x_\ell}\widetilde v_i^\varepsilon.
$$
All derivatives of $\widetilde f$ appearing below are globally bounded by the cut-off construction. Thus
\begin{align}\label{eq:nonlinear-residual}
&\|\nabla\cdot(\widetilde f(I_kP_N\widetilde v_i^\varepsilon)-\widetilde f(\widetilde v_i^\varepsilon))\|_{L_t^\infty L_x^{1+\delta}}\notag\\
&\qquad\lesssim\|Z_i\|_{L_t^\infty W_x^{1,1+\delta}}+\|\nabla \widetilde v_i^\varepsilon\|_{L_t^\infty L_x^\infty}\|Z_i\|_{L_t^\infty L_x^{1+\delta}} .
\end{align}
By Proposition~\ref{proputxn}, for $j=0,1$ and $1\le m\le k-1$,
\begin{align}
\|(I-I_k)\partial_x^j \widetilde v_i^\varepsilon\|_{L_t^\infty L_x^{1+\delta}}
&\lesssim\tau^m \ell_kk^{-m}\|\partial_t^m\partial_x^j\widetilde v_i^\varepsilon\|_{L_t^\infty L_x^{1+\delta}}\lesssim\varepsilon^{1-j-m-2\delta}\tau^m \ell_kk^{-m},\label{eq:time-approx-w}\\
\|(I-P_N)\partial_x^j \widetilde v_i^\varepsilon\|_{L_t^\infty L_x^{1+\delta}}
&\lesssim N^{-m}\|\partial_x^j \widetilde v_i^\varepsilon\|_{L_t^\infty W_x^{m,1+\delta}}
\lesssim\varepsilon^{1-j-m-2\delta}N^{-m}.\label{eq:space-approx-w}
\end{align}
Here and below, the sharp Fourier projection $P_N$ is used on $L^{1+\delta}(\mathbb T^d)$ only for a fixed $\delta>0$. Its operator norm may depend on $\delta$, and this dependence is included in the constants. In the final parameter choice, $\delta$ is fixed after the target exponent $\gamma$, so this dependence is absorbed into $C_\gamma$. Moreover,
$$
\|\nabla \widetilde v_i^\varepsilon\|_{L_t^\infty L_x^\infty}\lesssim\varepsilon^{-1-\delta}
$$
by \eqref{utxnbd} with $p=\infty$. Substituting \eqref{eq:time-approx-w}--\eqref{eq:space-approx-w} into \eqref{eq:nonlinear-residual} first gives the raw bound
$$
\|\nabla\cdot(\widetilde f(I_kP_N\widetilde v_i^\varepsilon)-\widetilde f(\widetilde v_i^\varepsilon))\|_{L_t^\infty L_x^{1+\delta}}\lesssim\varepsilon^{-m-3\delta}\tau^m\ell_kk^{-m}+\varepsilon^{-m-3\delta}N^{-m}.
$$
Indeed, the term containing $\|\nabla\widetilde v_i^\varepsilon\|_{L_t^\infty L_x^\infty}$ contributes
$$
\varepsilon^{-1-\delta}\cdot \varepsilon^{1-m-2\delta} =\varepsilon^{-m-3\delta}.
$$
Thus
\begin{equation}\label{eq:first-residual-term}
\|\nabla\cdot({\widetilde f}(I_kP_N\widetilde v_i^\varepsilon)-{\widetilde f}(\widetilde v_i^\varepsilon))\|_{L_t^\infty L_x^{1+\delta}}\lesssim\varepsilon^{-m-3\delta}\tau^m \ell_kk^{-m}+\varepsilon^{-m-3\delta}N^{-m}.
\end{equation}

For the second term in \eqref{eq:R-split}, we use the same temporal interpolation and spectral projection estimates, now applied to $\nabla\cdot\widetilde f(\widetilde v_i^\varepsilon)$. Proposition~\ref{proputxn}, with $g=\widetilde f_\ell$, implies
\begin{align}
\|(I-I_k)\nabla\cdot {\widetilde f}(\widetilde v_i^\varepsilon)\|_{L_t^\infty L_x^{1+\delta}}
&\lesssim \tau^m \ell_kk^{-m}\|\partial_t^m\nabla\cdot {\widetilde f}(\widetilde v_i^\varepsilon)\|_{L_t^\infty L_x^{1+\delta}}\lesssim \varepsilon^{-m-3\delta}\tau^m \ell_kk^{-m},\label{eq:flux-time}\\
\|(I-P_N)\nabla\cdot {\widetilde f}(\widetilde v_i^\varepsilon)\|_{L_t^\infty L_x^{1+\delta}}
&\lesssim N^{-m}\|\nabla\cdot {\widetilde f}(\widetilde v_i^\varepsilon)\|_{L_t^\infty W_x^{m,1+\delta}}\lesssim \varepsilon^{-m-3\delta}N^{-m}.\label{eq:flux-space}
\end{align}
It remains to estimate the commutator in \eqref{eq:R-split}. We write
$$
(\partial_tI_k-I_k\partial_t)P_N\widetilde v_i^\varepsilon =\partial_t\bigl((I_k-I)P_N\widetilde v_i^\varepsilon\bigr)-(I_k-I)P_N\partial_t\widetilde v_i^\varepsilon .
$$
The temporal commutator estimate established at the start of this subsection yields
\begin{align}
&\|(\partial_tI_k-I_k\partial_t)P_N\widetilde v_i^\varepsilon\|_{L_t^\infty L_x^{1+\delta}}
\notag\\
&\qquad\lesssim \tau^m \ell_kk^{-m} \|\partial_t^{m+1}P_N\widetilde v_i^\varepsilon
\|_{L_t^\infty L_x^{1+\delta}} \lesssim\varepsilon^{-m-3\delta}\tau^m \ell_kk^{-m}.
\label{eq:commutator-time}
\end{align}
Here $P_N$ is bounded on $L^{1+\delta}(\mathbb T^d)$, and the last inequality follows from Proposition~\ref{proputxn}. Combining \eqref{eq:R-split}, \eqref{eq:first-residual-term},\eqref{eq:flux-time}, \eqref{eq:flux-space}, and \eqref{eq:commutator-time} proves \eqref{eq:ri-bound}. By Lemma~\ref{cor:comparison-admissible}, $I_kP_N\widetilde v_i^\varepsilon\in\widehat{\mathcal A}_{i,k,N}$. Hence the penalized minimizing property \eqref{eq:min} gives
$$
\|g_{N,i}\|_{L^1(I_i\times\mathbb{T}^d)}+\|\phi(u_{N,i}^\varepsilon)\|_{L^1(I_i\times\mathbb{T}^d)}\le\|R_i\|_{L^1(I_i\times\mathbb{T}^d)}+\|\phi(I_kP_N\widetilde v_i^\varepsilon)\|_{L^1(I_i\times\mathbb{T}^d)}.
$$
By \eqref{rangepr}, the reference solution $v_i^\varepsilon$ takes values in $[u_-,u_+]$. Since $\phi(q)=\operatorname{dist}(q,[u_-,u_+])$, we have
$$
\phi(I_kP_N\widetilde v_i^\varepsilon) \le |I_kP_N\widetilde v_i^\varepsilon-v_i^\varepsilon|
\le |Z_i|+|\widetilde v_i^\varepsilon-v_i^\varepsilon|.
$$
Thus Lemma~\ref{lem:reference-comparison} and \eqref{eq:time-approx-w}--\eqref{eq:space-approx-w} give
\begin{align*}
\|\phi(I_kP_N\widetilde v_i^\varepsilon)\|_{L^1(I_i\times\mathbb T^d)} &\le \tau\|Z_i\|_{L^\infty(I_i;L^1)}+\tau E_i\\
&\le C\tau\varepsilon^{-m-3\delta} \bigl(\ell_k\tau^m k^{-m}+N^{-m}\bigr)+\tau E_i.
\end{align*}
Together with \eqref{eq:ri-bound}, this proves \eqref{eq:g-bound-opti}.
\end{proof}

\begin{proposition}\label{prop:local-error-tilde}
Assume that $\widetilde v_i^\varepsilon(T_i)$ satisfies the assumptions of Proposition~\ref{proputxn}. Assume also that $\tau\sim\varepsilon$. Then
\begin{equation}\label{eq:e-tilde-local}
\|u_{N,i}^\varepsilon-I_kP_N\widetilde v_i^\varepsilon\|_{L^\infty(I_i;L^1)}\le\tau E_i+C\tau \bigl(\varepsilon^{-m-3\delta}\tau^m \ell_kk^{-m}+\varepsilon^{-m-3\delta}N^{-m}\bigr),
\end{equation}
and consequently
\begin{equation}\label{eq:uN-tilde-local}
\|u_{N,i}^\varepsilon-\widetilde v_i^\varepsilon\|_{L^\infty(I_i;L^1)} \le\tau E_i+C\tau \bigl(\varepsilon^{-m-3\delta}\tau^m \ell_kk^{-m} +\varepsilon^{-m-3\delta}N^{-m}\bigr).
\end{equation}
\end{proposition}

\begin{proof}
Set
$$
e_i:=u_{N,i}^\varepsilon-I_kP_N\widetilde v_i^\varepsilon.
$$
Subtracting the equation defining $R_i$, \eqref{eq:ri}, from \eqref{eq:gNi}, we obtain
$$
\partial_te_i
+\nabla\cdot\bigl({\widetilde f}(u_{N,i}^\varepsilon) -{\widetilde f}(I_kP_N\widetilde v_i^\varepsilon)\bigr) -\varepsilon\Delta e_i=g_{N,i}-R_i .
$$
{
Since the cut-off flux has globally bounded derivative, we may write
$$
{\widetilde f}(u_{N,i}^\varepsilon)-{\widetilde f}(I_kP_N\widetilde v_i^\varepsilon)
=a_i(t,x)e_i,\qquad a_i(t,x):=\int_0^1 \widetilde f'\bigl(I_kP_N\widetilde v_i^\varepsilon+\theta e_i\bigr)\,d\theta,
$$
with $\|a_i\|_{L^\infty}\le \|\widetilde f'\|_{L^\infty(\mathbb R)}$.
}
Testing by the usual smooth approximation of $\sgn(e_i)$, integrating over $\mathbb{T}^d$, and passing to the limit gives the $L^1$-stability estimate with source:
$$
\|e_i(t)\|_{L^1} \le \|e_i(T_i)\|_{L^1} +\|g_{N,i}\|_{L^1(I_i\times\mathbb{T}^d)}
+\|R_i\|_{L^1(I_i\times\mathbb{T}^d)}.
$$
At the left endpoint,
$$
I_kP_N\widetilde v_i^\varepsilon(T_i)=P_N\widetilde v_i^\varepsilon(T_i)=u_{N,i}^\varepsilon(T_i),
$$
because the endpoint data in \eqref{eq:tilde-v-data} belong to the Fourier space $X$. Hence $e_i(T_i)=0$. The residual bound for $g_{N,i}$ is \eqref{eq:g-bound-opti}; the bound for $R_i$ follows directly from Lemma~\ref{lem:local-residual-opti}. This proves \eqref{eq:e-tilde-local}. Finally,
$$
\|u_{N,i}^\varepsilon-\widetilde v_i^\varepsilon\|_{L^1}\le\|e_i\|_{L^1}+\|(I-I_kP_N)\widetilde v_i^\varepsilon\|_{L^1},
$$
and the last term is bounded by \eqref{eq:time-approx-w} and \eqref{eq:space-approx-w} with $j=0$. Since $\tau\sim\varepsilon$, this term is absorbed by the right-hand side of \eqref{eq:e-tilde-local}. This proves \eqref{eq:uN-tilde-local}.
\end{proof}

\subsection{Global propagation and conclusion}

Combining Lemma~\ref{lem:reference-comparison} with Proposition~\ref{prop:local-error-tilde} gives, for $t\in I_i$,
\begin{align}\label{eq:global-recursion}
\|u_{N,i}^\varepsilon(t)-v_i^\varepsilon(t)\|_{L^1} &\le\|u_{N,i}^\varepsilon(t)-\widetilde v_i^\varepsilon(t)\|_{L^1} +\|\widetilde v_i^\varepsilon(t)-v_i^\varepsilon(t)\|_{L^1}\notag\\
&\le (1+\tau)E_i+C\tau\bigl(\varepsilon^{-m-3\delta}\tau^m \ell_kk^{-m}+\varepsilon^{-m-3\delta}N^{-m}\bigr),
\end{align}
where $E_i$ is defined above and $E_0\lesssim N^{-m}\varepsilon^{-m}$. Assuming for the moment that the initial regularity bounds, and hence the local-error constant, are uniform over the time slabs, discrete Gronwall applied to \eqref{eq:global-recursion}, using $(1+\tau)^M\le e^T$, yields
\begin{align}
\|u_N^\varepsilon(t)-v^\varepsilon(t)\|_{L^1} &\le C \fe^T \bigl(\varepsilon^{-m-3\delta}\tau^m \ell_kk^{-m}+\varepsilon^{-m-3\delta}N^{-m}\bigr)+C\fe^TN^{-m}\varepsilon^{-m}\label{eq:uN-v-global}\\
&\le C\bigl(\varepsilon^{-m-3\delta}\tau^m \ell_kk^{-m}+\varepsilon^{-m-3\delta}N^{-m}\bigr),
\qquad 0\le t\le T .\label{eq:uN-v-simple}
\end{align}

It remains to justify the regularity assumption used when applying Proposition~\ref{proputxn} to $\widetilde v_i^\varepsilon$. For the first slab, by the parabolic smoothing properties,
$$
\|\fe^{\varepsilon^2\Delta}u_0\|_{W^{l+1,1}} +\|\fe^{\varepsilon^2\Delta}u_0\|_{W^{l,\infty}}
\lesssim_l \varepsilon^{-l},\qquad l\ge0.
$$
Moreover, we do not use endpoint boundedness of the sharp projector $P_N$. Instead, $P_N\fe^{\varepsilon^2\Delta}$ is the convolution operator with the truncated heat kernel
$$
K_{\varepsilon,N}(x)=\sum_{|n|_\infty\le N}\fe^{-\varepsilon^2|n|^2}\fe^{inx}.
$$
Since $N=\varepsilon^{-1-\eta}$, the omitted heat-kernel tail is exponentially small. More precisely, for every multi-index $\alpha$, every $1\le p\le\infty$, and every fixed $r\ge1$,
$$
\|K_{\varepsilon,N}\|_{L^1(\mathbb{T}^d)} \le C,\qquad \|D^\alpha(I-P_N)\fe^{\varepsilon^2\Delta}f\|_{L^p} \le C_{\alpha,r,d}\varepsilon^{-|\alpha|-d}(N\varepsilon)^{-r}\|f\|_{L^p}.
$$
On the torus, this follows by estimating the $L^1(\mathbb T^d)$-norm of the
periodic tail kernel:
$$
\left\|\sum_{|n|_\infty>N}(\im n)^\alpha \fe^{-\varepsilon^2|n|^2}\fe^{\im nx}\right\|_{L^1(\mathbb T^d)}\lesssim\sum_{|n|_\infty>N}|n|^{|\alpha|}\fe^{-\varepsilon^2|n|^2}\lesssim\varepsilon^{-|\alpha|-d}(N\varepsilon)^{-r}.
$$
Since $N\varepsilon=\varepsilon^{-\eta}$, taking $r$ large enough makes $\varepsilon^{-d}(N\varepsilon)^{-r}$ harmless. For the $W^{l+1,1}$ bound we put one derivative on the initial data. Since $u_0\in\bv$, its distributional gradient is a finite measure, and the convolution estimate uses its total variation:
$$
\|(I-P_N)\fe^{\varepsilon^2\Delta}u_0\|_{W^{l+1,1}} \lesssim_l \varepsilon^{-l-d}(N\varepsilon)^{-r} \bigl(\|u_0\|_{L^1}+\tv(u_0)\bigr) \lesssim_l \varepsilon^{-l}.
$$
For the $W^{l,\infty}$ bound, we use
$$
\|(I-P_N)\fe^{\varepsilon^2\Delta}u_0\|_{W^{l,\infty}} \lesssim_l \varepsilon^{-l-d}(N\varepsilon)^{-r}\|u_0\|_{L^\infty} \lesssim_l \varepsilon^{-l}.
$$
Hence
\begin{align*}
\|P_N\fe^{\varepsilon^2\Delta}u_0\|_{W^{l+1,1}}&\le \|\fe^{\varepsilon^2\Delta}u_0\|_{W^{l+1,1}}
+\|(I-P_N)\fe^{\varepsilon^2\Delta}u_0\|_{W^{l+1,1}} \lesssim_l \varepsilon^{-l},\\
\|P_N\fe^{\varepsilon^2\Delta}u_0\|_{W^{l,\infty}}&\le \|\fe^{\varepsilon^2\Delta}u_0\|_{W^{l,\infty}}+\|(I-P_N)\fe^{\varepsilon^2\Delta}u_0\|_{W^{l,\infty}} \lesssim_l \varepsilon^{-l}.
\end{align*}
Hence, for the first slab,
$$
\widetilde v_0^\varepsilon(T_0)=P_N\fe^{\varepsilon^2\Delta}u_0
$$
satisfies the assumptions of Proposition~\ref{proputxn}. No endpoint range condition is needed for the numerical minimizer, because the nonlinearity in the discrete residual has been replaced by the globally smooth cut-off flux $\widetilde f$. We now close an induction over the time slabs with fixed constants. For each $l\ge0$, let $A_l$ be a constant independent of $i,N,k,\tau,\varepsilon$, chosen below to include the first-slab bound. At the beginning of slab $I_i$, the induction statement is
$$
\|\widetilde v_i^\varepsilon(T_i)\|_{W^{l+1,1}}+\|\widetilde v_i^\varepsilon(T_i)\|_{W^{l,\infty}}
\le A_l\varepsilon^{-l}, \qquad l\ge0,
$$
which is exactly the regularity assumption needed to apply Proposition~\ref{proputxn} on that slab. Once this is known on slab $I_i$, Lemma~\ref{lem:local-residual-opti} and Proposition~\ref{prop:local-error-tilde} give the local error estimate, and the propagation argument gives \eqref{eq:uN-v-simple} up to the end of slab $I_i$. It remains to show that this $L^1$-error estimate implies the same regularity bound at the beginning of slab $I_{i+1}$.

Thus assume that \eqref{eq:uN-v-simple} has been proved up to the end of slab $I_{i-1}$. We prove the above bound for $\widetilde v_i^\varepsilon(T_i)$. By the $L^1$-contraction of the heat semigroup (see also Lemma~\ref{lemcontra}), \eqref{eq:v-split-cont}, and the standard heat-kernel estimates,
\begin{align*}
\|\widetilde v_i^\varepsilon(T_i)\|_{W^{l+1,1}}&= \|\fe^{\varepsilon^2\Delta}u_{N,i-1}^\varepsilon(T_i)\|_{W^{l+1,1}}\\
&\le\|\fe^{\varepsilon^2\Delta}v_{i-1}^\varepsilon(T_i)\|_{W^{l+1,1}}+\|\fe^{\varepsilon^2\Delta}(u_{N,i-1}^\varepsilon(T_i)-v_{i-1}^\varepsilon(T_i))\|_{W^{l+1,1}}\\
&\lesssim_l\varepsilon^{-l}+\varepsilon^{-l-1}\|u_{N,i-1}^\varepsilon(T_i)-v_{i-1}^\varepsilon(T_i)\|_{L^1},
\end{align*}
and similarly
\begin{align*}
\|\widetilde v_i^\varepsilon(T_i)\|_{W^{l,\infty}}&\le\|\fe^{\varepsilon^2\Delta}v_{i-1}^\varepsilon(T_i)\|_{W^{l,\infty}}+\|\fe^{\varepsilon^2\Delta}(u_{N,i-1}^\varepsilon(T_i)-v_{i-1}^\varepsilon(T_i))\|_{W^{l,\infty}}\\
&\lesssim_l\varepsilon^{-l}+\varepsilon^{-l-d}\|u_{N,i-1}^\varepsilon(T_i)-v_{i-1}^\varepsilon(T_i)\|_{L^1}.
\end{align*}
Here the terms involving $v_{i-1}^\varepsilon$ are bounded by $C_l\varepsilon^{-l}$, since the exact splitting flow has uniformly bounded $W^{1,1}$ and $L^\infty$ norms, and the heat kernel supplies the remaining $l$ derivatives. By \eqref{eq:uN-v-simple},
$$
\|u_{N,i-1}^\varepsilon(T_i)-v_{i-1}^\varepsilon(T_i)\|_{L^1}\lesssim\varepsilon^{-m-3\delta}\tau^m \ell_kk^{-m}+\varepsilon^{-m-3\delta}N^{-m}.
$$
For this regularity bootstrap, fix $\eta>0$ and $\vartheta>0$, choose the integer $m$ so that $m\eta>d$, and then choose $\delta>0$ sufficiently small that
\begin{equation}\label{eq:bootstrap-parameter-conditions}
3\delta<\min\{m\eta-d,\ d\eta+(1+\eta)\vartheta\}.
\end{equation}
Take
$$
N=\varepsilon^{-1-\eta},\qquad \tau\sim\varepsilon,\qquad k\ge m+1,\qquad \ell_kk^{-m}\le N^{-d-\vartheta}.
$$
The two contributions to the global error satisfy
\begin{align*}
\varepsilon^{-m-3\delta}\tau^m \ell_kk^{-m}&\lesssim\varepsilon^{-3\delta}N^{-d-\vartheta}
=\varepsilon^{(1+\eta)(d+\vartheta)-3\delta},\\
\varepsilon^{-m-3\delta}N^{-m}&=\varepsilon^{m\eta-3\delta}.
\end{align*}
Both exponents are strictly larger than $d$ by \eqref{eq:bootstrap-parameter-conditions}. Thus, with
$$
\rho_\varepsilon:=
\varepsilon^{-m-3\delta}\tau^m \ell_kk^{-m}+\varepsilon^{-m-3\delta}N^{-m},
$$
we have $\rho_\varepsilon=o(\varepsilon^d)$.

To make the constants in this argument uniform, the heat-kernel bounds above provide fixed $B_l,D_l$ such that, for $i\ge1$,
$$
\varepsilon^l\left(\|\widetilde v_i^\varepsilon(T_i)\|_{W^{l+1,1}}+\|\widetilde v_i^\varepsilon(T_i)\|_{W^{l,\infty}}\right)\le B_l+D_l(\varepsilon^{-1}+\varepsilon^{-d})E_i.
$$
Choose $A_l\ge B_l+2D_l$, also large enough for the first slab, before starting the induction. With these bounds fixed, the local estimate has a fixed constant $C_{\rm loc}$; for the fixed derivative order $m$, it uses only finitely many of the $A_l$. The accumulated estimate is then
$$
E_i\le \fe^T\bigl(E_0+T C_{\rm loc}\rho_\varepsilon\bigr) \le C_*\rho_\varepsilon
$$
on every initial sequence of slabs on which the induction hypothesis holds. Here $E_0\lesssim N^{-m}\varepsilon^{-m}\le\rho_\varepsilon$, and $C_*$ is independent of the number of completed slabs and of the discretization parameters. Since $\rho_\varepsilon/\varepsilon^d\to0$, choose a single $\varepsilon_0\le1$ such that $C_*\rho_\varepsilon\le\varepsilon^d$ for $0<\varepsilon\le\varepsilon_0$. Then, using $d\ge1$,
$$
B_l+D_l(\varepsilon^{-1}+\varepsilon^{-d})E_i \le B_l+D_l(\varepsilon^{d-1}+1)\le B_l+2D_l\le A_l.
$$
Thus the next initial value satisfies the same fixed regularity bounds. This closes the induction and justifies the uniform constant in \eqref{eq:uN-v-simple}.

\begin{proof}[Proof of Theorem~\ref{thm}]
Combining \eqref{eq:v-u-error} and \eqref{eq:uN-v-simple}, we obtain
\begin{equation}\label{eq:final-error-before-choice}
\|u_N^\varepsilon(t)-u(t)\|_{L^1}\le C\varepsilon^{1/2} +C\varepsilon^{-m-3\delta}\tau^m \ell_kk^{-m}+C\varepsilon^{-m-3\delta}N^{-m}, \qquad 0\le t\le T .
\end{equation}
Given any fixed $\gamma>0$, first choose $\eta>0$ sufficiently small and set $N=\varepsilon^{-1-\eta}$, so that
$$
\frac1{2(1+\eta)}\ge \frac12-\gamma.
$$
Then $\varepsilon^{1/2}\le N^{-1/2+\gamma}$. We also take $\tau\sim\varepsilon$. The spatial term becomes
$$
\varepsilon^{-m-3\delta}N^{-m} =\varepsilon^{m\eta-3\delta} =N^{-\frac{m\eta-3\delta}{1+\eta}} .
$$
Fix $\vartheta>0$, choose $m$ so that $m\eta>d$, and choose $\delta>0$ satisfying \eqref{eq:bootstrap-parameter-conditions}. These choices ensure both the regularity bootstrap and the final error rate: indeed,
$$
\frac{m\eta-3\delta}{1+\eta}>\frac{d}{1+\eta} \ge\frac1{2(1+\eta)}\ge\frac12-\gamma.
$$
The spatial term is then bounded by $C_\gamma N^{-1/2+\gamma}$, where the dependence of intermediate constants on this fixed $\delta$ is absorbed into $C_\gamma$. The temporal interpolation term is
$$
\varepsilon^{-m-3\delta}\tau^m \ell_kk^{-m}\lesssim \varepsilon^{-3\delta}\ell_kk^{-m},
$$
Choose $k\ge m+1$ so that $\ell_kk^{-m}\le N^{-d-\vartheta}$. Then
$$
\varepsilon^{-3\delta}\ell_kk^{-m} \le N^{-d-\vartheta+3\delta/(1+\eta)},\qquad
d+\vartheta-\frac{3\delta}{1+\eta}>\frac{d}{1+\eta}\ge\frac12-\gamma,
$$
again by \eqref{eq:bootstrap-parameter-conditions}. Hence the temporal term is bounded by $C_\gamma N^{-1/2+\gamma}$ as well. Consequently,
\begin{align*}
\|u_N^\varepsilon(t)-u(t)\|_{L^1}\le C\varepsilon^{1/2}+C_\gamma N^{-1/2+\gamma}\le C_\gamma N^{-1/2+\gamma}, \qquad 0\le t\le T .
\end{align*}
This proves Theorem~\ref{thm}.
\end{proof}

\section{Numerical experiments}\label{sectionnumexp}

In this section, we provide numerical illustrations for the proposed VVMM Fourier method. The computations are a practical realization of the residual-minimizing method analyzed above: the continuous residual minimization is approximated by quadrature, a smoothed $L^1$-type residual, and an iterative optimization procedure. The experiments report convergence behavior, CPU/GPU time, and representative shock profiles in one and two spatial dimensions.

On each time interval, the temporal variable is discretized by interpolation at Chebyshev nodes defined in \eqref{cheby}.

Since global optimality cannot in general be guaranteed, we find that the initialization of the unknown values at these nodes also plays a significant role in practice. Here, we employ a fourth-order Taylor expansion based on the first four temporal derivatives obtained recursively from \eqref{vis} as the initial guess. 

The solution is computed by minimizing a weighted residual functional. More precisely, at each Chebyshev node, the PDE residual is measured by the smoothed $L^1$-type quantity
\begin{equation}\label{numres}
\frac{R(u)^2}{\sqrt{\delta_{\rm num}^2+R(u)^2}}, \qquad R(u)=u_t+\nabla\cdot f(u)-\varepsilon\Delta u
\end{equation}
and the total objective is defined as the corresponding quadrature-weighted average over all interpolation nodes. The gradient of this objective with respect to the optimization variables is obtained by a variational calculation, and spatial derivatives are evaluated spectrally by FFT. Moreover, a penalty term 
$$
\phi(u)=\max\{0,u_--u,u-u_+\},
$$
is also included at each Chebyshev node to discourage violations of the maximum principle bounds, where $u_-$ and $u_+$ are real numbers defined by the initial data (see also~\eqref{u-u+}).

The optimization is performed by gradient descent with an adaptive step size and a one-dimensional line search. Iterations are terminated when the reduction of the objective produced by the line search becomes negligible (up to $\mathcal{O}(\delta_{\rm num})$).

Overall, the numerical results are consistent with the theoretical $L^1$ convergence behavior and illustrate stable approximation of both rarefaction waves and shock waves. No pronounced Gibbs-type oscillations are observed in the reported tests.

\subsection{One-dimensional case}

\begin{figure}
\begin{center}
\subfigure[]{\includegraphics[width=0.48\textwidth]{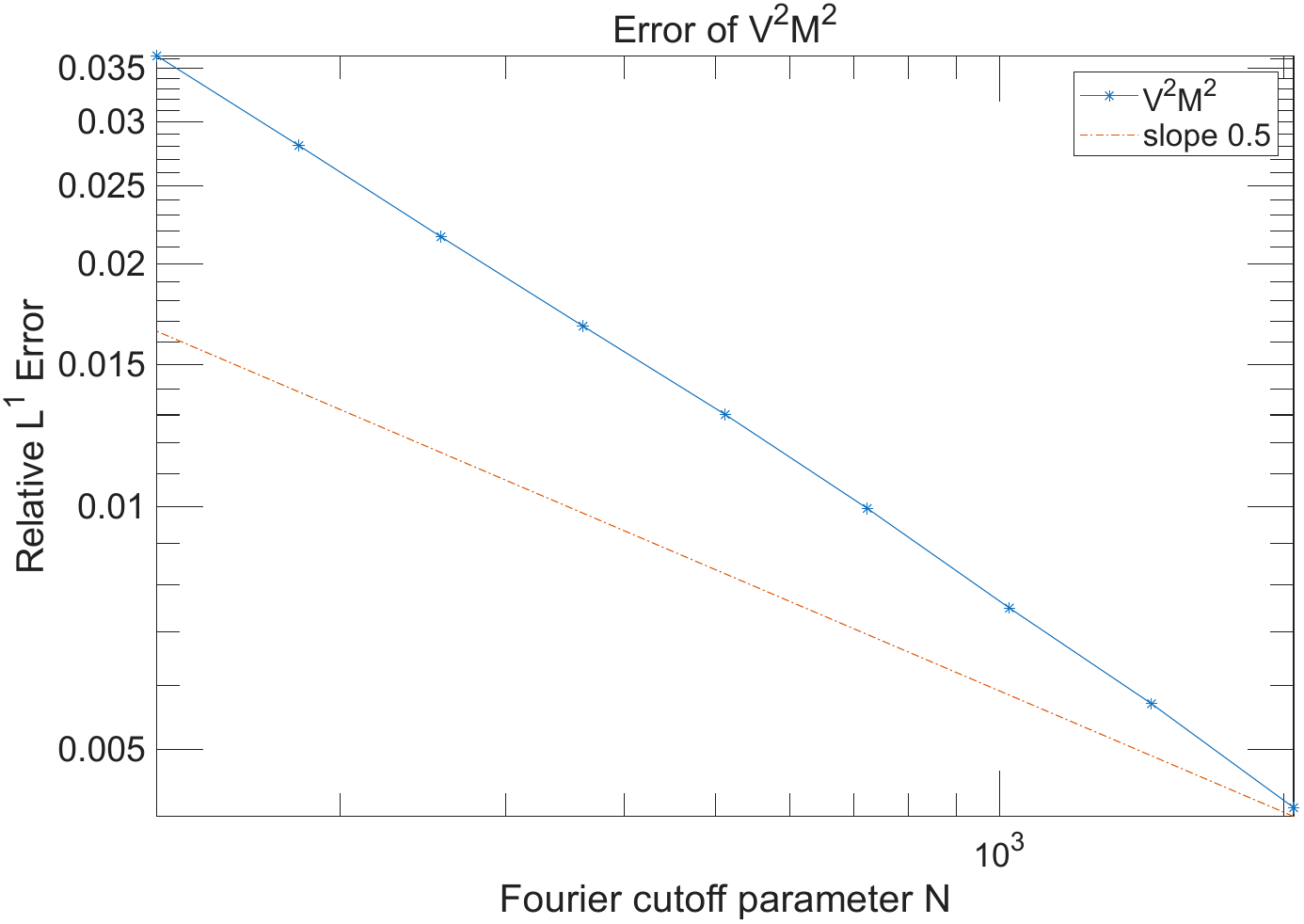}}
\subfigure[]{\includegraphics[width=0.48\textwidth]{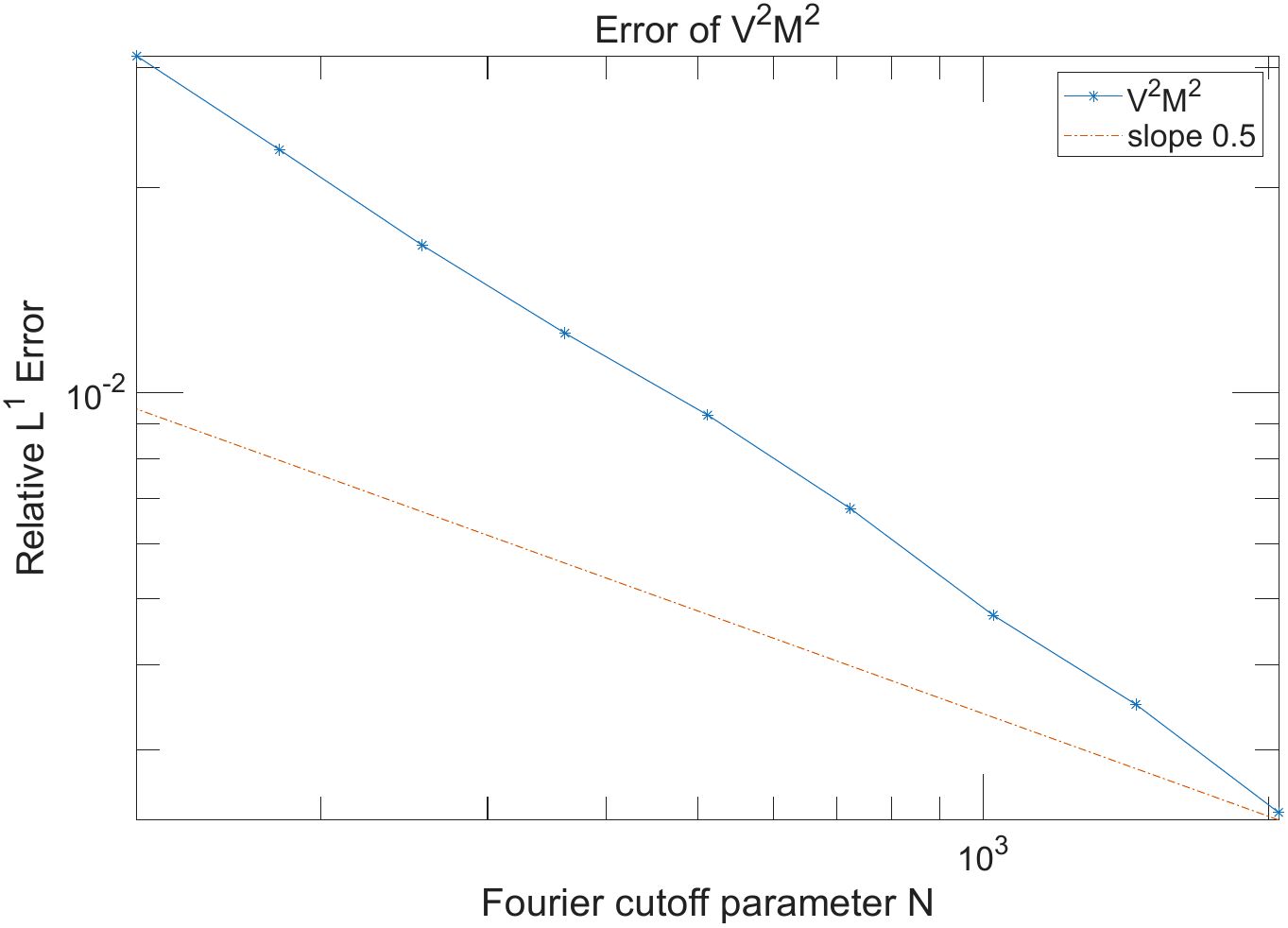}}
\subfigure[]{\includegraphics[width=0.48\textwidth]{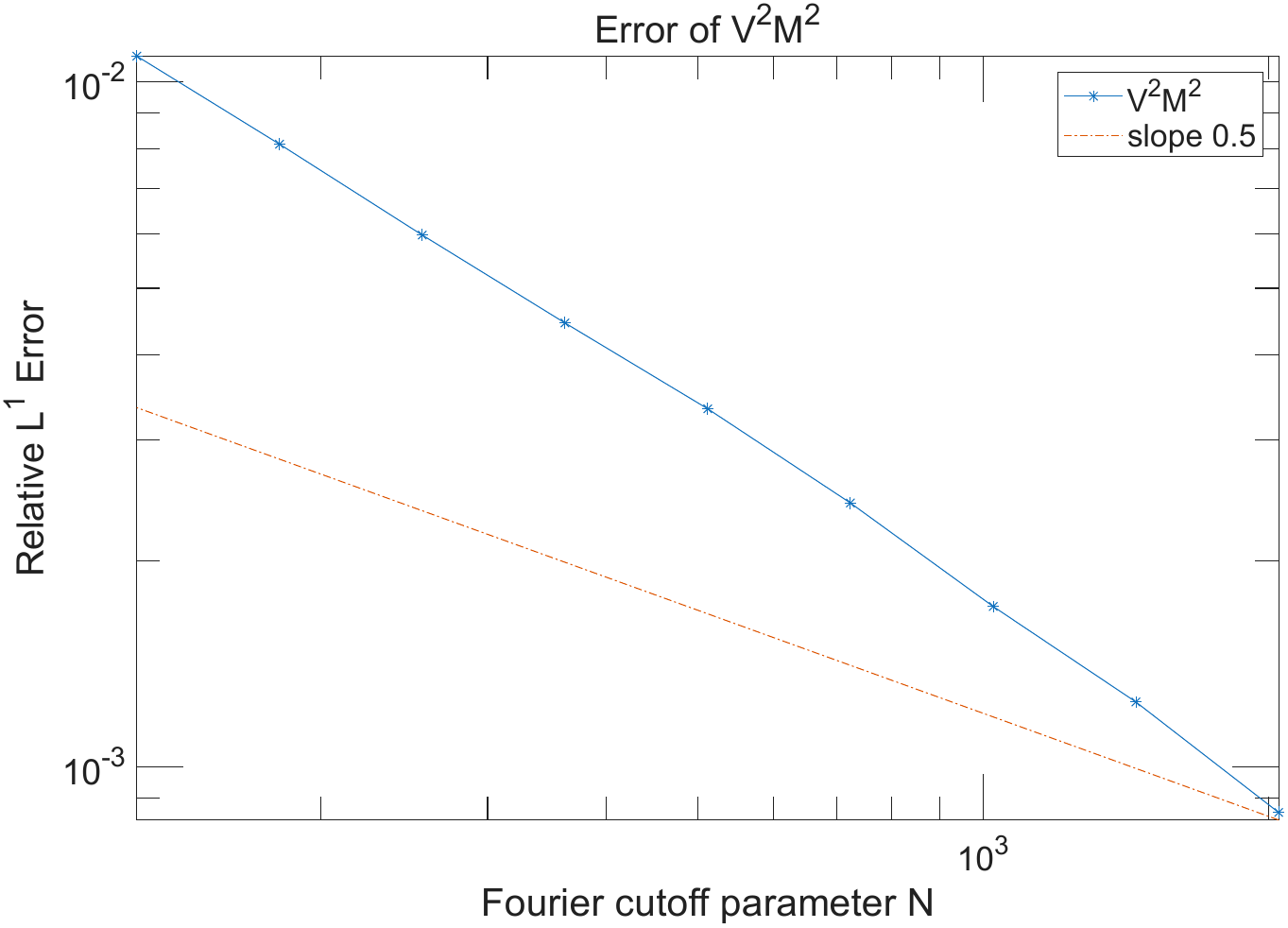}}
\subfigure[]{\includegraphics[width=0.48\textwidth]{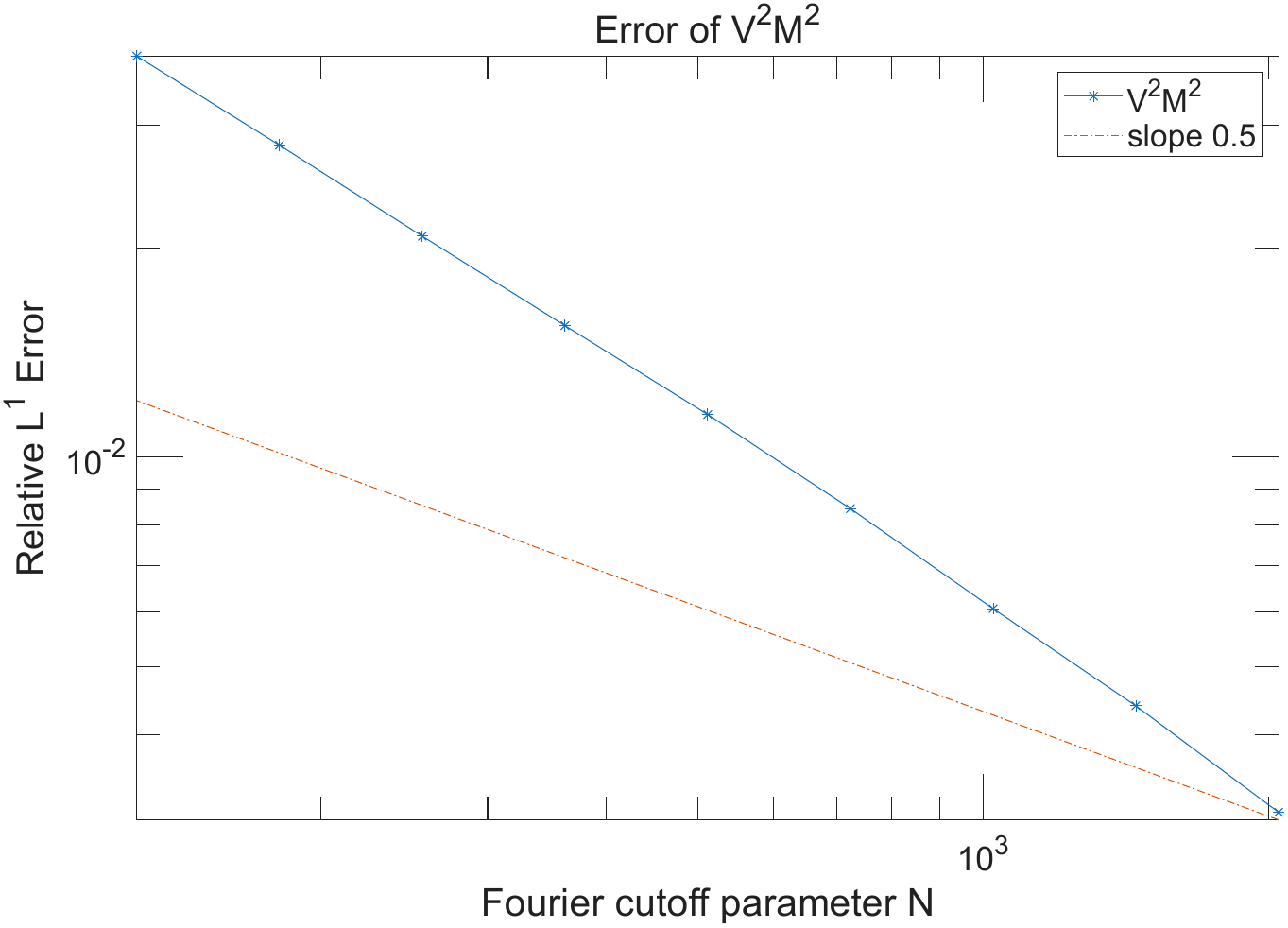}}
\end{center}
\caption{The $L^1$ global error of the one-dimensional conservation laws against the Fourier cutoff parameter $N$ for different fluxes and initial data. (a): example \ref{eg1}; (b): example \ref{eg2}; (c): example \ref{eg3}; (d): example \ref{eg4}.\label{fig1}}
\end{figure}

\begin{figure}
\begin{center}
\subfigure[]{\includegraphics[width=0.48\textwidth]{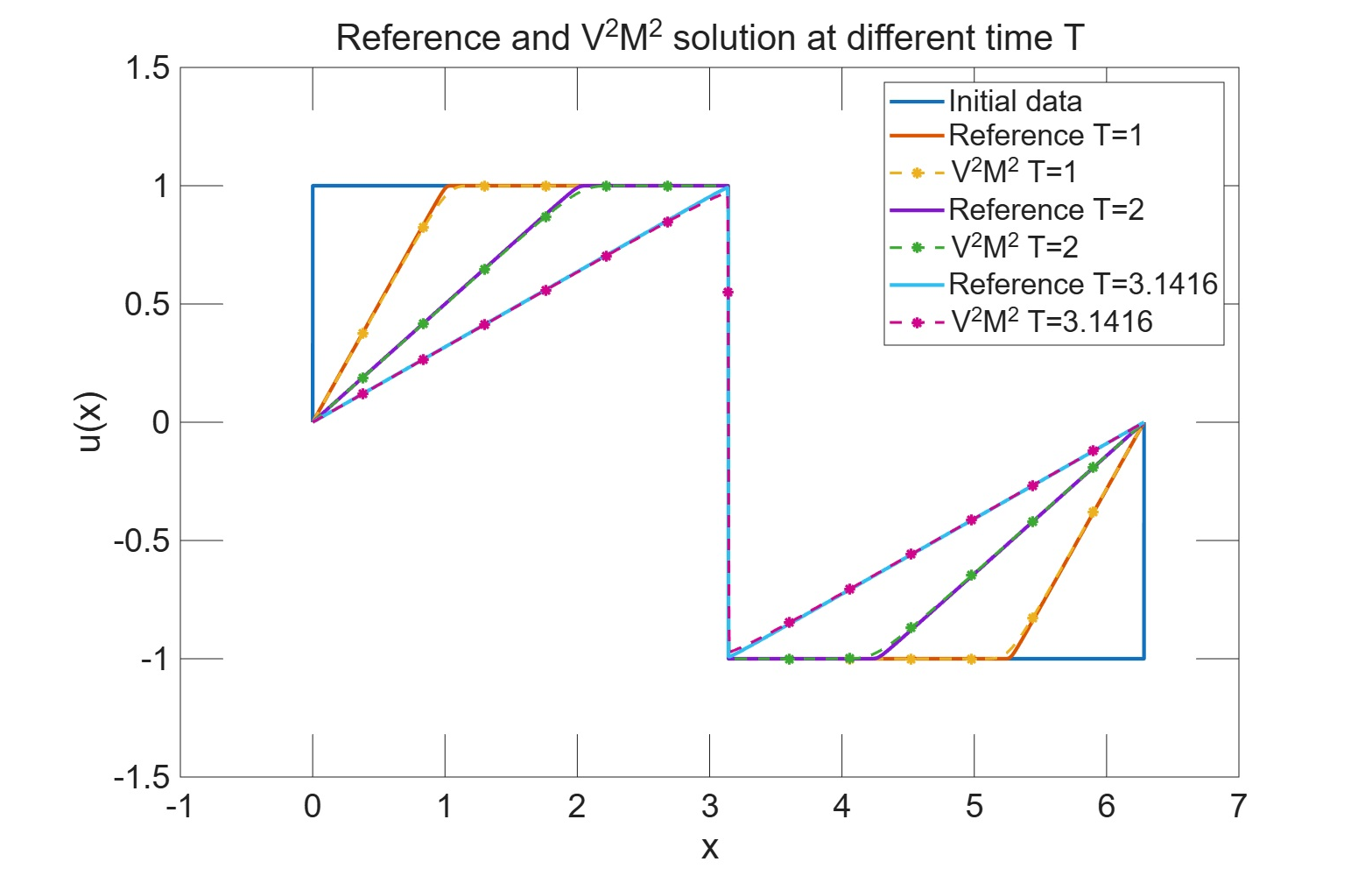}}
\subfigure[]{\includegraphics[width=0.48\textwidth]{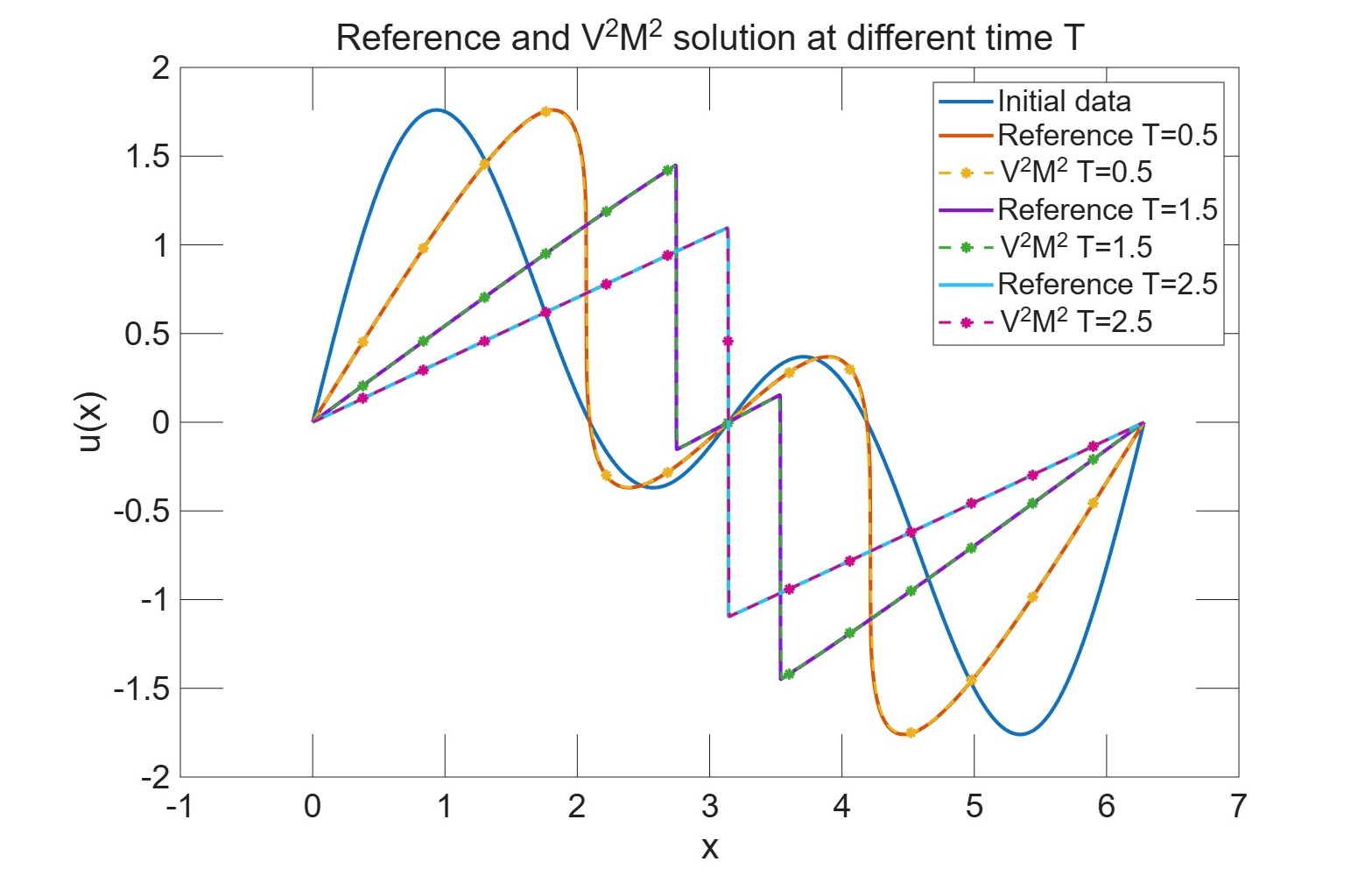}}
\subfigure[]{\includegraphics[width=0.48\textwidth]{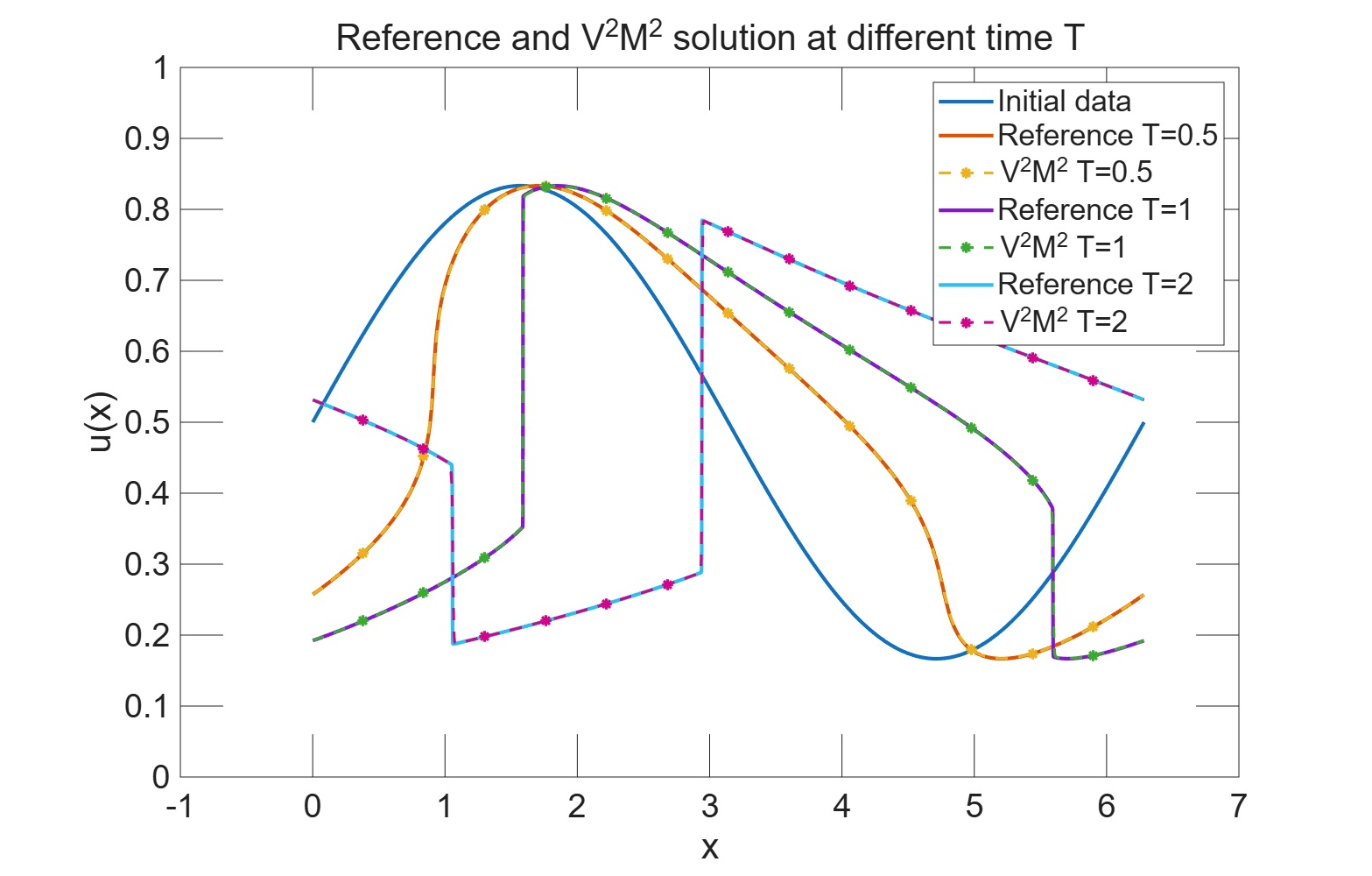}}
\subfigure[]{\includegraphics[width=0.48\textwidth]{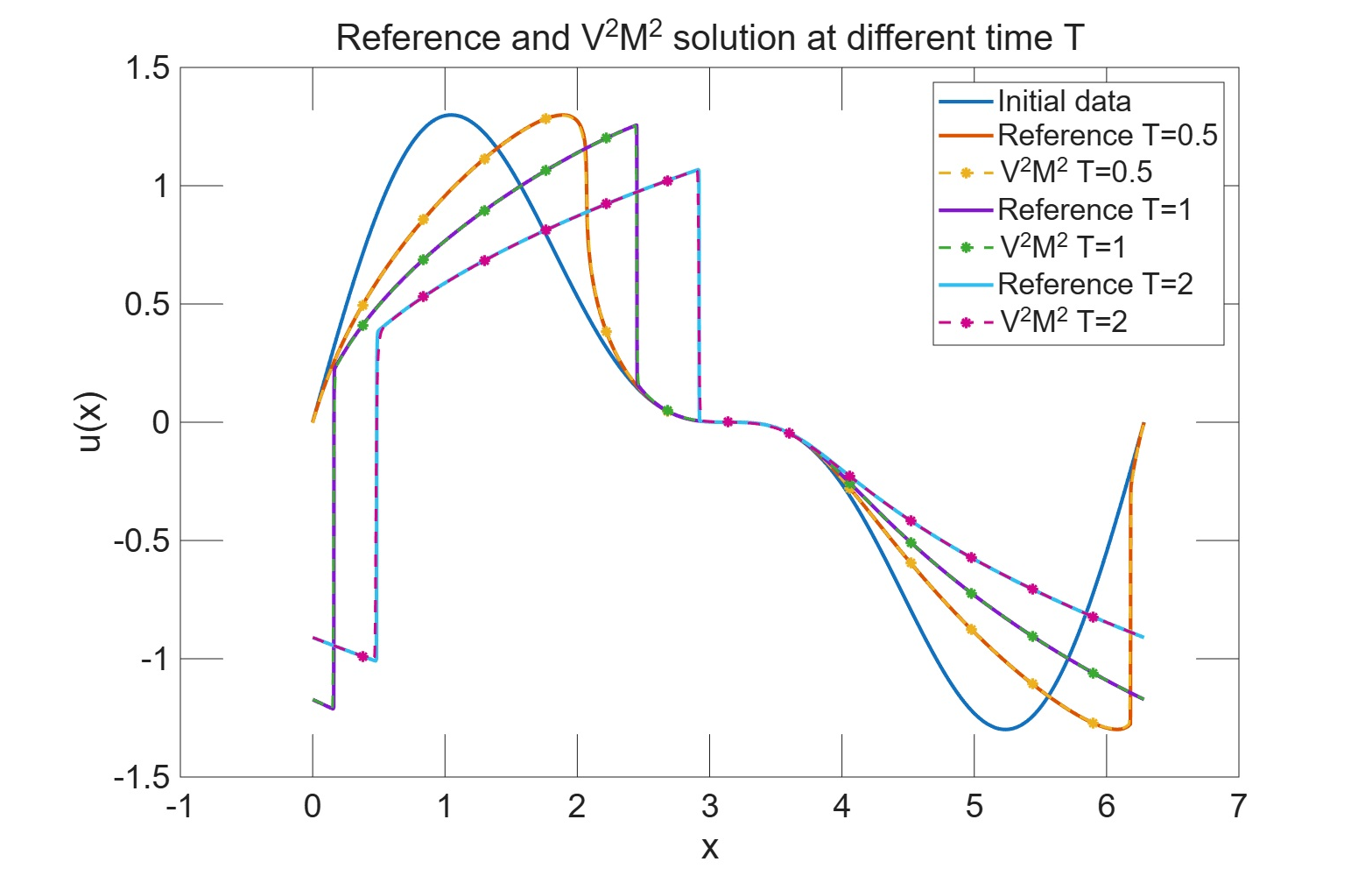}}
\end{center}
\caption{Comparison plots of the one-dimensional conservation laws between the reference solution and the numerical solution with $N=2^{11}$ for different fluxes and initial data. (a): example \ref{eg1}; (b): example \ref{eg2}; (c): example \ref{eg3}; (d): example \ref{eg4}.\label{fig2}}
\end{figure}

We take several different fluxes with different initial data as examples. For all these four cases, we take Chebyshev order $k=7$, and choose the Fourier cutoff parameter $N$ between $2^7$ and $2^{11}$ for test. For the reference solution, we employ the global Lax-Friedrichs scheme on a much finer spatial grid corresponding to the resolution parameter  $N=2^{13}$ (i.e., $2^{14}$ spatial grids)  and CFL number $c=0.4$. We set $\tau=\varepsilon=(2N)^{-0.85}$.  The parameter $\delta_{\rm num}$ given in \eqref{numres} is set to be $(2N)^{-1.5}$, and the optimization tolerance is $5\delta_{\rm num}$.

The proposed fluxes and initial data are listed as follows:
\begin{enumerate}[label=(\roman*)]
\item Flux $f(u)=u^2/2$ and initial data $u_0=\sgn(\sin x)$.\label{eg1}
\item Flux $f(u)=u^2/2$ and initial data $u_0=\sin x+\sin2x$.\label{eg2}
\item Flux $f(u)=\frac{u^2}{u^2+0.5(1-u)^2}$ and initial data $u_0=\frac13\sin x+\frac12$.\label{eg3}
\item Flux $f(u)=u^3/3$ and initial data $u_0=\sin x+\frac12\sin2x$.\label{eg4}
\end{enumerate}

For the Burgers' example \ref{eg1}, on the torus $[0,2\pi)$, there is supposed to be a rarefaction wave generated at the border $x=0$ and a shock wave at $x=\pi$. At time $T=\pi$, the rarefaction wave and the shock wave meet. We measured the relative $L^1$ global error (i.e., the global error divided by the $L^1$ norm of the reference solution) at $T=1$ against the Fourier cutoff parameter $N$ and compared our VVMM numerical solution of $N=2^{11}$ at time $T=1,2,\pi$ with the reference solution, see Figure~\ref{fig1}(a) and Figure~\ref{fig2}(a), respectively.

We also tested Burgers' equation with a smooth initial data $u_0=\sin x+\sin2x$, which is denoted as the example \ref{eg2}. In this case, two shock waves are generated at the same time point $T=\frac{16}{33}$. After that, the two shocks move towards each other and meet at $x=\pi$ and $t\approx 2.1$. We measured the $L^1$ global error at $T=1$ against the Fourier cutoff parameter $N$ and compared our VVMM numerical solution of $N=2^{11}$ at time $T=0.5,1.5,2.5$ with the reference solution, see Figure~\ref{fig1}(b) and Figure~\ref{fig2}(b), respectively.

We further tested two different nonconvex fluxes, denoted as examples \ref{eg3} and \ref{eg4}, respectively, where compound waves may occur, namely combinations of rarefaction waves and shock waves. Such waves do not occur for the convex Burgers flux. The plots of global error against the Fourier cutoff parameter $N$ for these two examples are given in Figure~\ref{fig1}(c) and (d), respectively, while the comparisons between our VVMM numerical solution with $N=2^{11}$ at times $T=0.5,1,2$ and the reference solution are given in Figure~\ref{fig2}(c) and (d), respectively.

Figures~\ref{fig1} and \ref{fig2} show that the practical implementation approximates the reference entropy solution well for different fluxes and initial data. The observed convergence rates are consistent with the theoretical almost one-half order rate and are slightly higher in these tests.

We further tested the CPU time for different $N$. The computational costs of all four examples are similar. As an example, we report the CPU time against the spatial resolution and time step size for example \ref{eg1} at time $T=1$, see Figure~\ref{figtime}(a). The one-dimensional numerical experiments were run on a laptop equipped with an Intel Core i7-14700 CPU.

\begin{figure}[h]
\begin{center}
\subfigure[]{\includegraphics[width=0.48\textwidth]{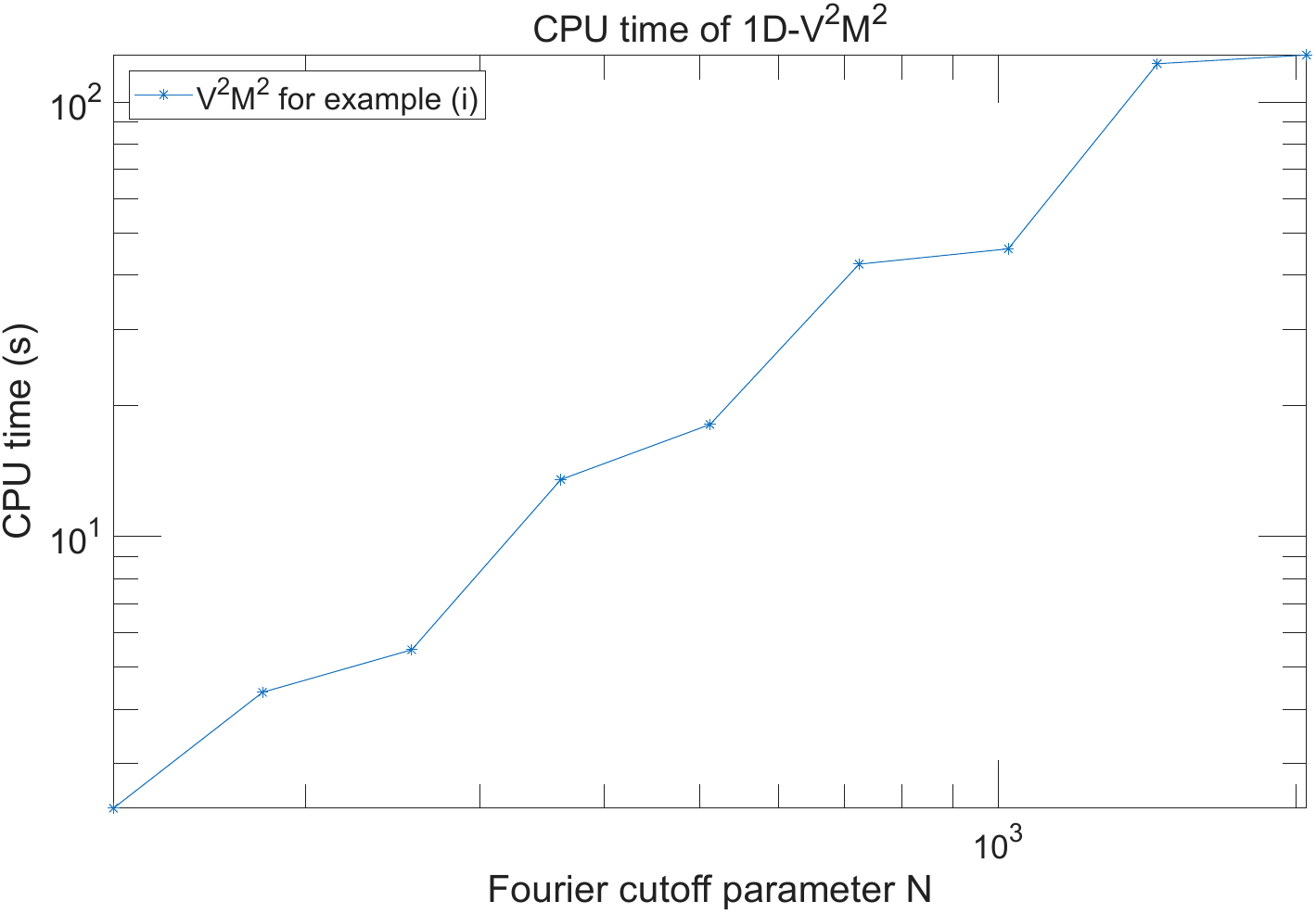}}
\subfigure[]{\includegraphics[width=0.48\textwidth]{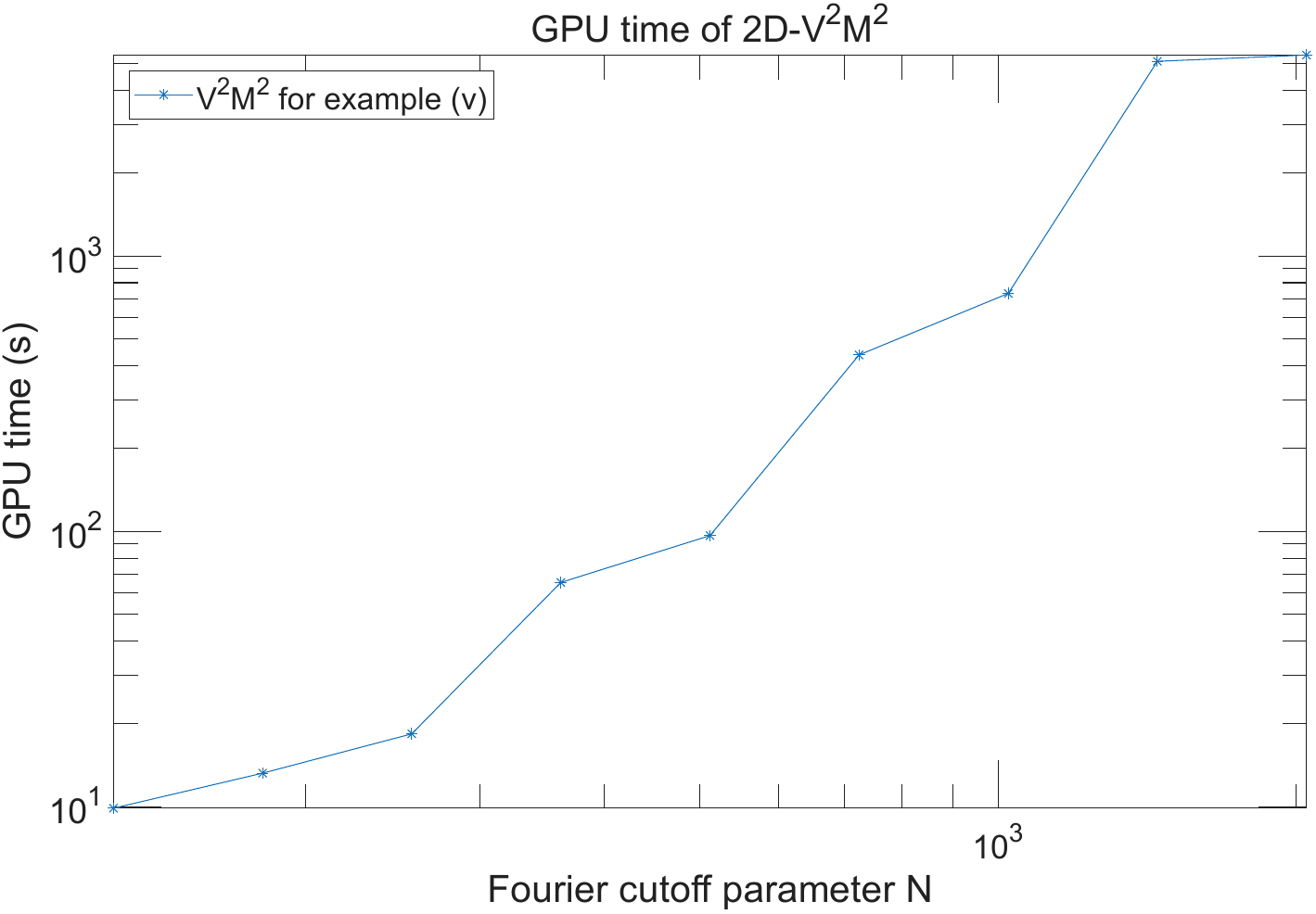}}
\end{center}
\caption{CPU/GPU time of our VVMM method (see Section~\ref{sectionmethod}) for one- and two-dimensional conservation laws with final time $T=1$. (a): CPU time for the one-dimensional example~\ref{eg1}; (b) GPU time for the two-dimensional example~\ref{eg5}.\label{figtime}}
\end{figure}

\subsection{Two-dimensional case}

\begin{figure}
\begin{center}
\subfigure[]{\includegraphics[width=0.48\textwidth]{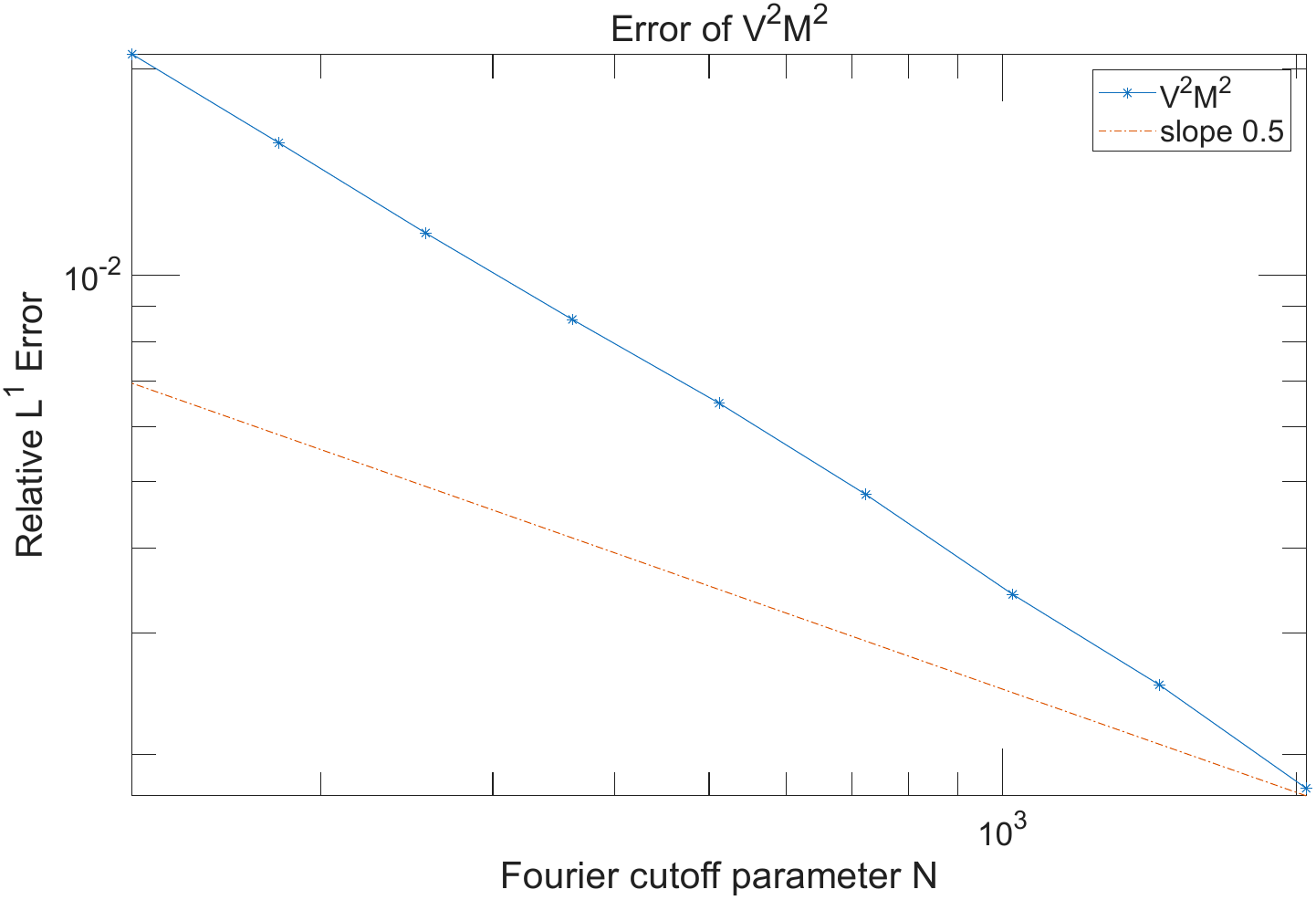}}
\subfigure[]{\includegraphics[width=0.48\textwidth]{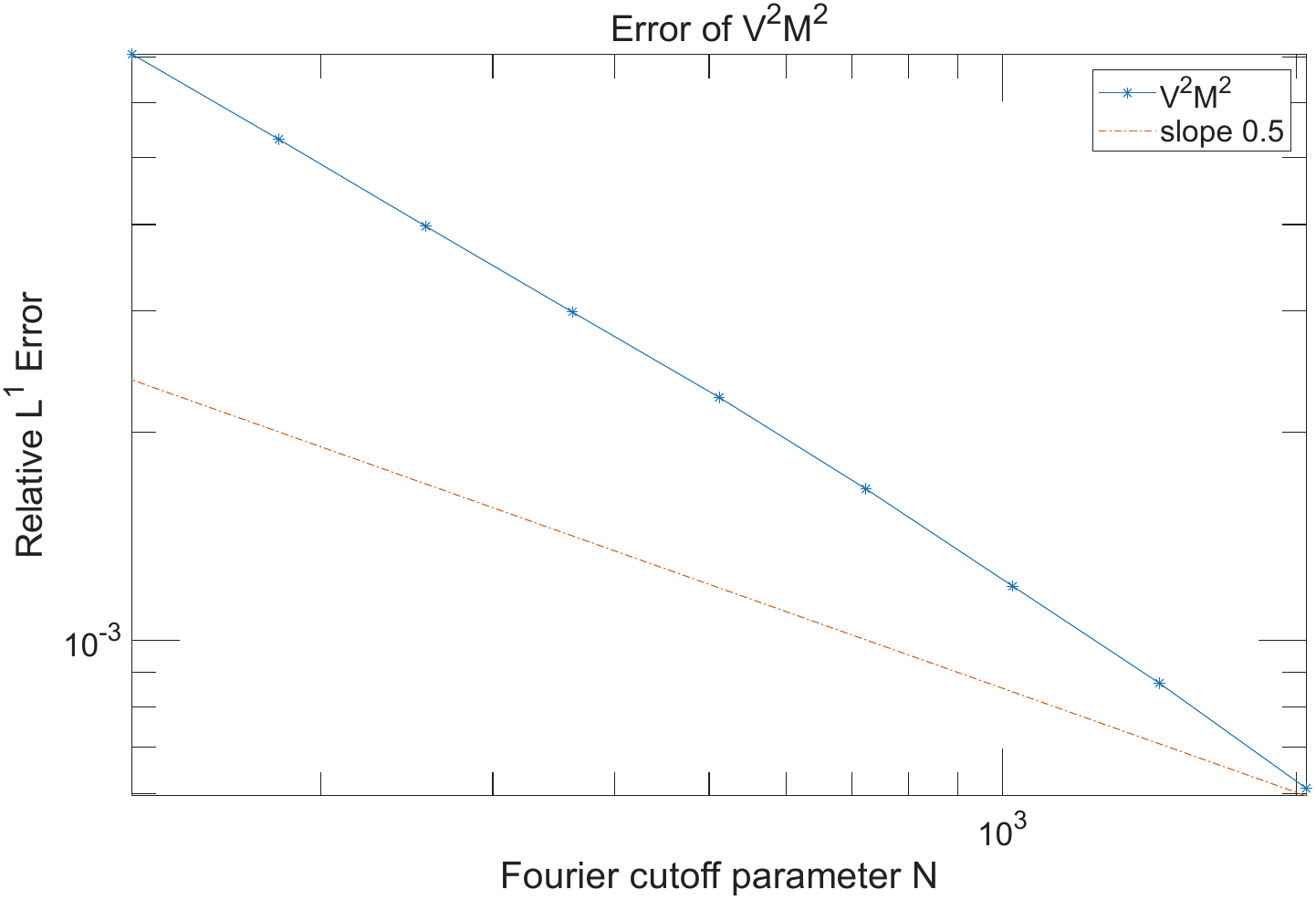}}
\end{center}
\caption{The $L^1$ global error of the two-dimensional conservation laws against the Fourier cutoff parameter $N$ for different fluxes and initial data. (a): example \ref{eg5}; (b): example \ref{eg6}.\label{fig3}}
\end{figure}

\begin{figure}
\begin{center}
\subfigure[]{\includegraphics[width=0.48\textwidth]{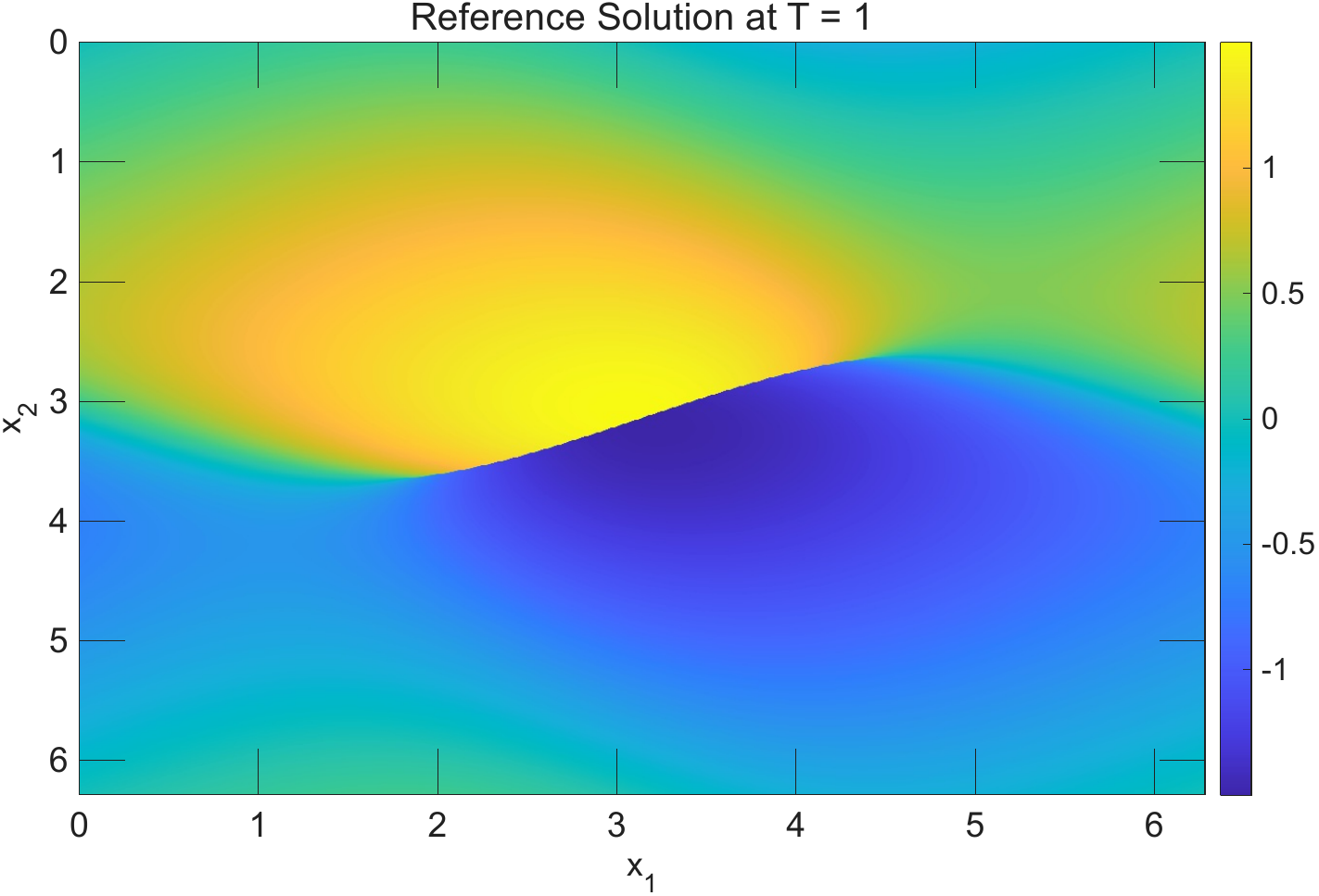}}
\subfigure[]{\includegraphics[width=0.48\textwidth]{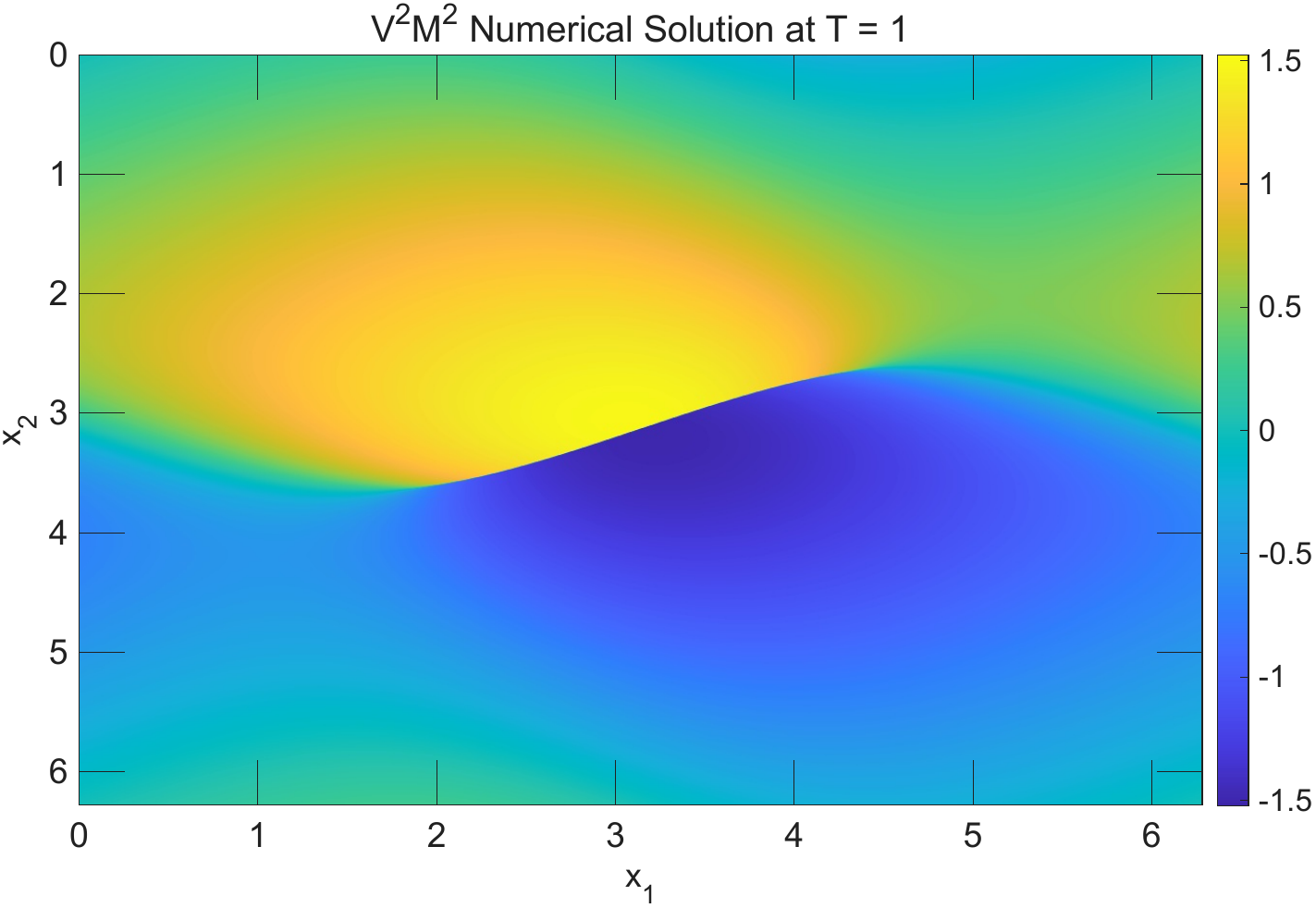}}
\subfigure[]{\includegraphics[width=0.48\textwidth]{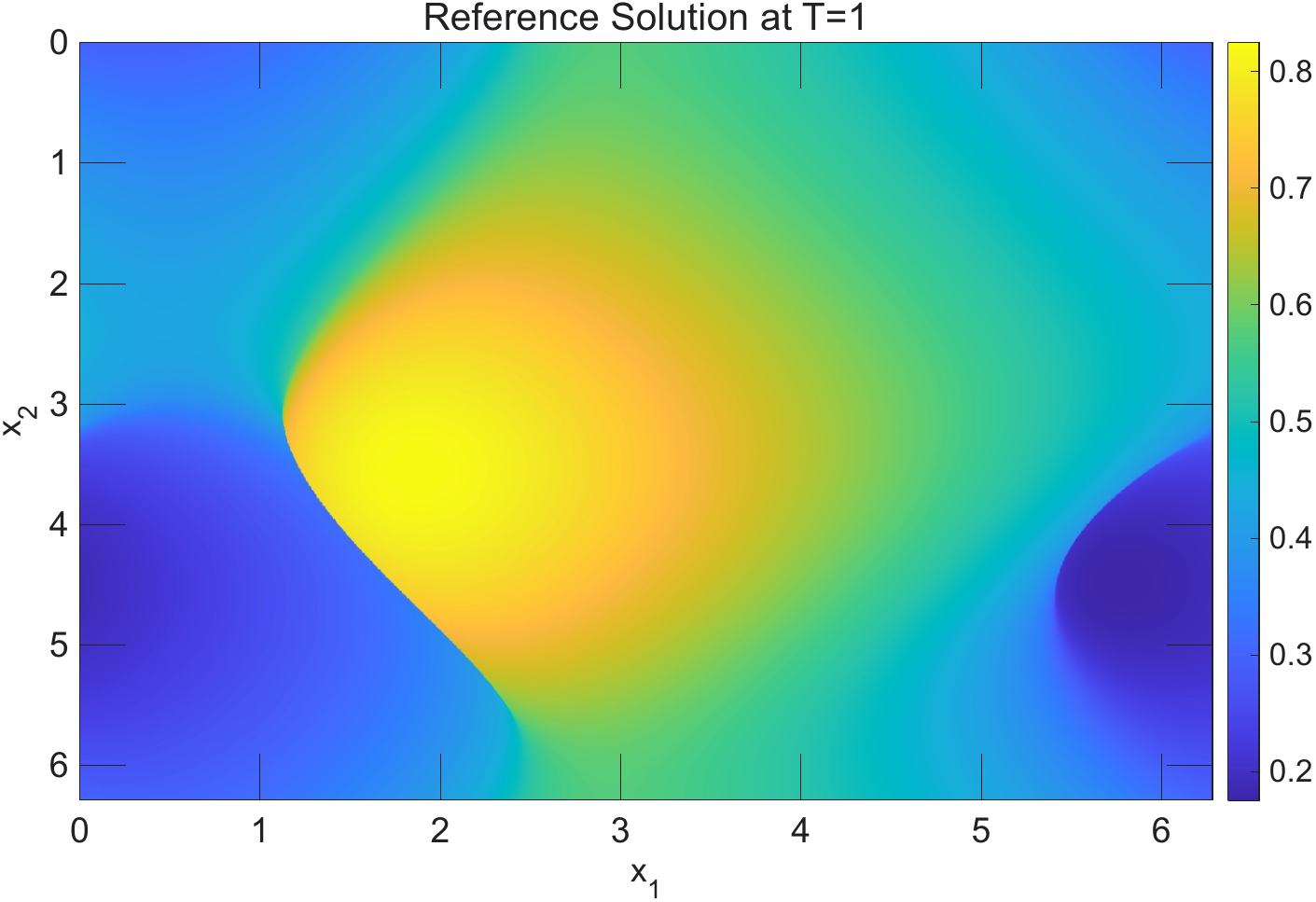}}
\subfigure[]{\includegraphics[width=0.48\textwidth]{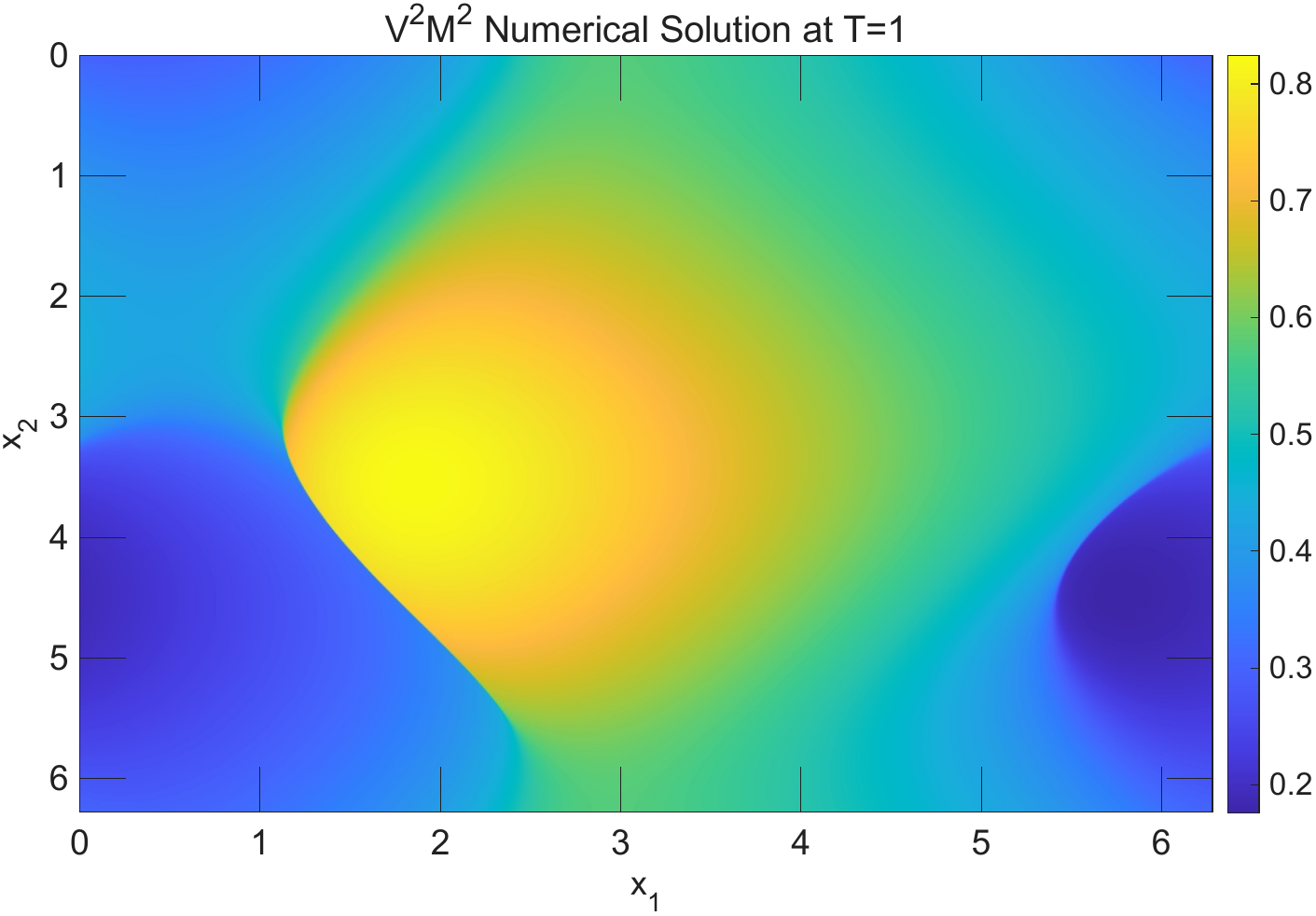}}
\end{center}
\caption{Comparison plots of the two-dimensional conservation laws between the reference solution and the numerical solution with $N=2^{11}$ for different fluxes. (a): reference solution for example \ref{eg5}; (b): VVMM numerical solution for example \ref{eg5}; (c): reference solution for example \ref{eg6}; (d): VVMM numerical solution for example \ref{eg6}.\label{fig4}}
\end{figure}

For the two-dimensional case, we take two different fluxes as examples. For both cases, we take Chebyshev order $k=7$, choose the Fourier cutoff parameter $N$ between $2^7$ and $2^{11}$ for test, and set $\tau=\varepsilon=(2N)^{-0.85}$.  We set the parameter $\delta_{\rm num}$ given in \eqref{numres} to be $(2N)^{-1.5}$, and the optimization tolerance is set to be $5\delta_{\rm num}$. For the reference solution, we employ the dimensional splitting method
$$\left\{\begin{aligned}
&u_t+\partial_1f_1(u)=0\\
&u_t+\partial_2f_2(u)=0,
\end{aligned}\right.$$
where both subequations are solved by global Lax-Friedrichs scheme on a fine spatial grid corresponding to the resolution parameter $N=2^{13}$ and CFL number $c=0.4$. The convergence order of this splitting method is $\frac12$, see \cite[Theorem~5.19]{Holdenetal2010}.

The proposed fluxes and initial data are listed as follows:
\begin{enumerate}[label=(\roman*),start=5]
\item Flux $f(u)=\big(u^2/2,u^2/2\big)$ and initial data $u_0=\frac12\sin x_1+\sin x_2$.\label{eg5}
\item Flux $f(u)=\big(\frac{u^2}{u^2+0.5(1-u)^2},\frac{u^2(u^2-0.5(1-u)^2)}{u^2+0.5(1-u)^2}\big)$ and initial data $u_0=\frac15\sin x_1+\frac18\sin x_2+\frac12$.\label{eg6}
\end{enumerate}

We first take Burgers' equation, denoted as example \ref{eg5}. Note that since
$$
u_t+uu_{x_1}+uu_{x_2}=u_t+\sqrt{2}uu_{x_2'},
$$
where the transformation $(x_1',x_2')=((x_1-x_2)/\sqrt2,(x_1+x_2)/\sqrt2)$ reduces the equation to a family of one-dimensional Burgers equations in $x_2'$, parametrized by $x_1'$.

For this example, we exhibit the relative $L^1$ error plot at $T=1$ against the Fourier cutoff parameter $N$ between $2^7$ and $2^{11}$ and compared our VVMM numerical solution of $N=2^{11}$ at time $T=1$ with the reference solution, see Figure~\ref{fig3}(a) and Figure~\ref{fig4}(a)(b), respectively.

We further tested the weighted Buckley--Leverett equation, denoted by example \ref{eg6}. This example uses a two-dimensional nonconvex flux and is included to test the method beyond convex scalar fluxes and beyond equations evolving along a single fixed spatial direction. For this example, we plotted the $L^1$ global error at $T=1$ against the Fourier cutoff parameter $N$ and compared our VVMM numerical solution of $N=2^{11}$ at time $T=1$ with the reference solution, see Figure~\ref{fig3}(b) and Figure~\ref{fig4}(c)(d), respectively.

Figures~\ref{fig3} and~\ref{fig4} show good agreement between the VVMM approximation and the reference solution for both two-dimensional examples. The observed convergence rates are again consistent with the theoretical almost one-half order rate and are slightly higher in these tests.

The experiments use a fixed degree $k=7$, whereas the parameter condition in Theorem~\ref{thm} requires $k$ to grow with $N$. The numerical smoothing parameter $\delta_{\rm num}$ is distinct from the regularity-loss exponent $\delta$; quadrature, residual smoothing, and optimization errors are not covered by the theorem. We observe that further increasing $k$ or decreasing $\delta_{\rm num}$ leads to only marginal improvements in the numerical results.

We further tested the GPU time for different $N$. The computational costs of both cases are similar. As an example, Figure~\ref{figtime}(b) reports the GPU time against the spatial resolution and time step size for example \ref{eg5} at time $T=1$. The two-dimensional numerical experiments were run on an Nvidia A100 cluster.

\end{document}